\documentclass[reqno]{amsart}
 
\usepackage{amssymb}
\usepackage{mathrsfs}
\usepackage{graphicx}
\usepackage{url}
\usepackage{enumitem}
\usepackage{lgreek}

\def\S{\mathcal S}

\newcommand{\be}{\begin{equation}}
\newcommand{\ee}{\end{equation}}
\newcommand{\ba}{\begin{align}}
\newcommand{\ea}{\end{align}}

\newcommand{\abs}[1]{\lvert#1\rvert}

\DeclareMathOperator{\ind}{ind}

\newtheorem{theorem}{Theorem}[section]

\newtheorem{proposition}[theorem]{Proposition}
\newtheorem{lemma}{Lemma}[section]
\newtheorem{question}[lemma]{Question}
{\begin{list}{}{%
\settowidth{\labelwidth}{\textsf{{\it #1.}}}%
\setlength{\labelsep}{4mm}%
\setlength{\leftmargin}{\labelwidth}%
\addtolength{\leftmargin}{\labelsep}%
}}%
{\end{list}}

{\begin{list}{}{%
\settowidth{\labelwidth}{\textsf{{\it #1.}}}%
\setlength{\labelsep}{2mm}%
\setlength{\leftmargin}{\labelwidth}%
\addtolength{\leftmargin}{\labelsep}%
\addtolength{\leftmargin}{4mm}%
\setlength{\itemsep}{6pt}%
\setlength{\listparindent}{0pt}%
\setlength{\topsep}{3pt}%
}}%
{\end{list}}

\usepackage{xcolor}
\newcommand{\upd}[1]{#1}

\title[Circulant determinants]{Prime power order circulant determinants}

\author[M.~J. Mossinghoff]{Michael J.  Mossinghoff}
\address{Center for Communications Research, Princeton, NJ, USA}
\email{m.mossinghoff@idaccr.org}

\author[C. Pinner]{Christopher Pinner}
\address{ Department of Mathematics\\
         Kansas State University\\
         Manhattan, KS 66506, USA}
\email{pinner@math.ksu.edu}

\date{\today}

\begin{document}

\begin{abstract}
Newman showed that for primes $p\geq 5$ an integral circulant determinant  of prime power order $p^t$ cannot take the value $p^{t+1}$
once $t\geq 2.$
We show that many other values are also excluded.
In particular, we show that $p^{2t}$ is the smallest power of $p$ attained for any $t\geq 3$, $p\geq 3.$
We demonstrate the complexity involved by giving a complete description of the $25\times 25$ and $27\times 27$ integral circulant determinants.
The former case involves a partition of the primes that are $1\bmod5$ into two sets, Tanner's  \textit{perissads} and  \textit{artiads}, which were later characterized by E. Lehmer.
\end{abstract}

\maketitle

\section{Introduction}\label{secIntroduction}

The problem of characterizing the values taken by  an integral $n\times n$ integral circulant determinant
$$ D(a_0,\ldots ,a_{n-1})=\begin{pmatrix} a_0 & a_1 & \cdots & a_{n-1}\\
a_{n-1} & a_0 & \cdots & a_{n-2}\\ 
\vdots  & \vdots &  \ddots  &\vdots \\
a_1 & a_2 & \cdots & a_{0} \end{pmatrix} $$
was suggested by Olga Taussky-Todd  \cite{OTT} at  the meeting of the American Mathematical Society in Hayward, California, in April 1977.
These are the integral group determinants for the cyclic group of order $n$, which we denote by $\mathbb Z_n,$ and we
write
$$ \S(\mathbb Z_n) =\{ D(a_0,\ldots ,a_{n-1})\; :\; a_0,\ldots ,a_{n-1}\in \mathbb Z\}. $$
It is not known whether the integral group determinants determine the group, as with the group determinant  polynomial \cite{Formanek}. 
For a polynomial $F$ in $\mathbb Z[x]$ we write
\begin{equation*}\label{measuredefn}
M_n(F)=\prod_{z^n=1} F(z)
\end{equation*}
and observe, using Dedekind's factorization of the group determinant or via eigenvalues (see  for example \cite{Conrad,Frobenius,Newman1}), that
\be \label{measureform} D(a_0,\ldots ,a_{n-1})=M_n\left( \sum_{j=0}^{n-1} a_jx^j \right). \ee
This connects the circulant determinant to Lind's \cite{Lind} generalization of the Mahler measure, in this case to the group $\mathbb Z_n$ and polynomials in  $\mathbb Z[x]/\langle x^n-1\rangle$, see for example \cite{Norbert2} (or \cite{Lalin,luck}
for an alternative approach). The form \eqref{measureform}  makes it easy to see that $\S(\mathbb Z_n)$ is closed under multiplication.
A complete description of the integral group determinants for groups of order less than $16$ can be found in  \cite{15,smallgps}. Other groups have been considered in  \cite{dihedral,pgroups,dilum,Norbert,Heisenberg,dicyclic,Pigno1,S4,Yamaguchi3}, where even determining 
the smallest non-trivial integral group determinant for the group, the counterpart of the classical Lehmer Problem \cite{Lehmer}, can become complicated.
For cyclic groups Laquer \cite{Laquer} and Newman \cite{Newman1} showed that 
\begin{equation*}\label{coprime}
\{ m\in \mathbb Z\;:\; \gcd(m,n)=1\}\subseteq \S(\mathbb Z_n)
\end{equation*}
and $n^2\mathbb Z\subseteq \S(\mathbb Z_n),$ along with a divisibility condition for an $m\in \S(\mathbb Z_n)$:
\be \label{div} p\mid m,\; p^{t}\parallel n \;\; \Rightarrow \; p^{t+1}\mid m. \ee
Here and throughout $p$ denotes a prime.
For $n=p$  and $n=2p,$ $p$ odd,  they showed that \eqref{div} was both necessary and sufficient:
\begin{align*} \S(\mathbb Z_p) & =\{p^am\; :\; \gcd(m,p)=1,\; a=0 \text{ or } a\geq 2\},\\
 \S(\mathbb Z_{2p}) & =\{2^ap^bm\; :\; \gcd(m,2p)=1,\; a=0 \text{ or } a\geq 2,\; b=0 \text{ or } b\geq 2 \}. 
\end{align*}
Newman \cite{Newman1} showed this is also the case for $n=9$,
\[
\S(\mathbb Z_{3^2})=\{3^am\; : \; \gcd(3,m)=1,\; a=0 \text{ or } a\geq 3\}.
\]
In general however this condition is not sufficient. 
Newman \cite{Newman2} showed that  $p^3\not\in \S(\mathbb Z_{p^2})$ for any $p\geq 5.$ By extension one obtains that $p^{t+1}\not\in \S(\mathbb Z_{p^t})$ for any $t\geq 2$ when $p\geq 5.$
Newman's result shows that the smallest power of $p$ in $\S(\mathbb Z_{p^2})$ is $3^3$ for $p=3$ and $p^4$ for $p\geq 5$. For general $\mathbb Z_{p^t}$, $t\geq 3,$ we show that an adaptation of Newman's approach excludes not only $p^{t+1}$ but additional powers of $p$, enough to determine that the smallest achieved power of $p$ is always $p^{2t}$.
We  can also put restrictions on the multiples of smaller powers of $p$ that are achieved.

\begin{theorem}\label{thmpowers}
Suppose that $G=\mathbb Z_{p^t}$ with $p\geq 5$ and  $t\geq 2$
or $p=3$ and $t\geq 3.$  Then $p^{2t}\mathbb Z \subset  \S(G),$ but
$p^j\not \in \S(G)$ for $t+1\leq j \leq 2t-1.$
Moreover, suppose that   $p^{j}m\in \S(G)$ with $p\nmid m$ and $t+1\leq j\leq 2t-1.$ 
\begin{enumerate}[label=(\alph*)]
\item\label{Zpta}
If $p\geq 5,$ or $p=3$ and $j>t+1,$ then  there exist $2t-j$  prime powers $q_i^{a_i} \equiv 1 \bmod p^i,$ each $a_i$ odd, such that $\prod_{i=1}^{2t-j} q_i^{a_i} \mid m$.
\item\label{Zptb}
If $p=3$ and $j=t+1,$ then there are $t-2$ prime powers with  $q_i^{a_i} \equiv 1 \bmod 3^{i},$ each $a_i$ odd, such that $\prod_{i=2}^{t-1} q_i^{a_i} \mid m$.
\end{enumerate}
\end{theorem}

We should note that $p^{t+\ell}\parallel M_{p^t}(x-1+p^{\ell})$ for all $\ell\geq 1$
so that there will be some  $mp^j,$ $p\nmid m,$  with $t+1\leq j\leq 2t-1.$
While restrictive, we do not expect the conditions on $m$ in Theorem~\ref{thmpowers} to be sufficient.
We illustrate this by characterizing the determinants for $\mathbb Z_{25}$ and $\mathbb Z_{27}$, the smallest unresolved cases of odd prime power cyclic group determinants.
Even for these small cases the situation becomes quite complicated. We shall use the polynomial form \eqref{measureform}
and split the  product over the $n=p^t$th roots of unity into $t+1$ integer norms:
$$ M_{p^t}(F) =F(1) \prod_{k=1}^t N_k(F),\;\; \;\;\;N_k(F):= \prod_{\stackrel{j=1}{p\nmid j}}^{p^k} F(\omega_k^j),\;\;\;\;\; \omega_k:=e^{2\pi i/p^k}. $$
From Theorem \ref{thmpowers} it follows that if $5^3m\in\mathcal{S}(\mathbb Z_{25})$ with $5\nmid m$, then $m$ must be divisible by a prime $q \equiv 1 \bmod 5$.
Similarly, if $3^4m \in\mathcal{S}(\mathbb Z_{27})$ and $3\nmid m$, then $m$ must be divisible by a prime $q\equiv 1 \bmod 9$ or the cube of a prime $q=4$ or $7 \bmod 9$, and if $3^5m\in\mathcal{S}(\mathbb Z_{27})$ and $3\nmid m$, then $m$ must be divisible by a prime $q\equiv 1 \bmod 3.$
However, not all $m$ of these forms can be achieved in the corresponding set $\mathcal{S}(G)$.
For $\S(\mathbb Z_{25})$ we get $5^3qm,$ $5\nmid m,$ for $134$ of the $163$ primes $q\equiv 1\bmod 5$ less than $5000$, but not $5^3q$ (nor any $5^3 q^k$) for the remaining $29$.
Similarly for $\S(\mathbb Z_{27})$ we get $3^4qm,$ $3\nmid m,$ for $76$ of the $109$ primes $q\equiv 1 \bmod 9$ less than $5000$,  but not $3^4q^k$ for the remaining $33$ such primes $q$, nor for any primes $q\equiv 4$ or $7 \bmod 9$.
For this group we also get $3^5qm,$ $3\nmid m$,  for $251$ of the $330$ primes $q\equiv 1 \bmod 3$ less than $5000$, but not $3^5q^k$ for the other $79$.

For the analysis of $\mathcal{S}(\mathbb Z_{25})$, we partition the primes $q\equiv 1 \bmod 5$ into two sets:
\begin{equation}\label{eqn5Types12}
\begin{split}
\mathcal{T}^{(5)}_1&=\{ q=N_1(1+(x-1)^3h(x))\; : \;h\in \mathbb Z[x],\; 5\nmid h(1)\},\\
\mathcal{T}^{(5)}_2&=\{ q=N_1(1+(x-1)^3h(x))\; : \; h\in \mathbb Z[x],\;5\mid h(1)\}.
\end{split}
\end{equation}
We refer to these respectively as \textit{Type~1} or \textit{Type~2} (with respect to the prime $5$).
We show in Lemma~\ref{lemNormRep} that $\mathcal{T}^{(5)}_1$ and $\mathcal{T}^{(5)}_2$ are disjoint, and that every $q\equiv 1 \bmod 5$ lies in one of these two sets.
A somewhat similar division of primes was needed to describe the integral group determinants for  $\mathbb Z_{12}$ and $\mathbb Z_2 \times \mathbb Z_{6}$ in \cite{smallgps} and $\mathbb Z_{15}$ in \cite{15}.

\upd{The partition  $\mathcal{T}^{(5)}_1$ and $\mathcal{T}^{(5)}_2$ first arose in some work of Tanner well over a century ago \cite{Tanner1,Tanner2}.
Tanner called these sets the \textit{perissads} and \textit{artiads} respectively (after \begin{greek}peritt\'os\end{greek} and \begin{greek}\'artios\end{greek}, Greek for odd and even).
The artiads appear in fact as the sequence A001583 in the OEIS \cite{OEIS}.
In Lemma~\ref{artiad} below we show that our sets \eqref{eqn5Types12} coincide with  Tanner's original definition (given in terms of the coefficients of the Jacobi function).
In 1966, E.~Lehmer \cite{ELehmer} characterized the artiads as those primes $q\equiv 1\bmod 5$ where the roots of $x^2-x-1$ mod $q$ are quintic residues mod $q$, or in terms of Fibonacci numbers that $q\mid F_{(q-1)/5}.$
In terms of Dickson's representation \cite{Dickson1} for primes $q=1\bmod5$,
\be \label{Dicksonform}  16q=x^2+50u^2+50v^2+125w^2, \;\; x\equiv 1\bmod 5,\;\; xw=(v-2u)^2-5u^2, \ee
Lehmer showed that they are the $q$ for which $5\mid w.$}

With this we can state our characterization of $\mathcal{S}(\mathbb Z_{25})$.
In particular, one may classify the primes $q\equiv 1\bmod 5$ as perissads or artiads depending on whether $5^3q$ is a $25\times 25$ integer circulant determinant.

\begin{theorem}\label{thmZ25} In addition to $5^4\mathbb Z$ and the integers coprime to $5$, the elements of $\S(\mathbb Z_{25})$ are those of the form $5^3qm$ where $\gcd(m,5)=1$ and  $q\equiv 1 \bmod 5$ is a perissad.
\end{theorem}

While this gives a complete description of the values achieved, it is not immediately obvious whether a prime $q\equiv 1 \bmod 5$ is Type~1 or Type~2. 
For a given $q$ one could factor it in $\mathbb Z[\omega_1]$ and use the straightforward algorithm of Lemma~\ref{primenorm} to decide its type.
Assuming randomness mod $5$ in the key coefficient in \eqref{eqn5Types12}, we might anticipate achieving $5^3q$ for about $4/5$ of the primes $q\equiv 1 \bmod 5$, and indeed empirical results support this heuristic (see also Lehmer \cite[Cor.~1]{ELehmer}).
Of the $163$ primes $q\equiv 1 \bmod 5$ up to $5000$, for convenience we list here the $134$ that are Type~1 (the perissads):
\begin{quote}
11, 31, 41, 61, 71, 101, 131, 151, 181, 191, 241, 251, 271, 311, 
331, 401, 431, 491, 541, 571, 601, 631, 641, 661, 701, 751, 761, 811, 
821, 911, 941, 971, 1021, 1051, 1061, 1091, 1171, 1181, 1201, 1231, 
1291, 1301, 1321, 1361, 1381, 1451, 1471, 1481, 1531, 1571, 1621, 
1721, 1741, 1801, 1811, 1831, 1861, 1901, 1931, 2011, 2081, 2111, 
2131, 2141, 2161, 2251, 2281, 2311, 2341, 2351, 2371, 2381, 2411, 
2441, 2521, 2531, 2551, 2621, 2671, 2711, 2731, 2741, 2791, 2801, 
2851, 2861, 2971, 3011, 3041, 3061, 3121, 3181, 3191, 3221, 3271, 
3301, 3331, 3361, 3371, 3391, 3461, 3491, 3511, 3541, 3581, 3631, 
3671, 3691, 3701, 3761, 3821, 3881, 3911, 3931, 4001, 4051, 4091, 
4111, 4201, 4211, 4231, 4241, 4261, 4271, 4421, 4451, 4561, 4591, 
4721, 4801, 4831, 4861, 4931, 4951,
\end{quote}
and the $29$ of Type~2 (the artiads):
\begin{quote}
211, 281, 421, 461, 521, 691, 881, 991, 1031, 1151, 1511, 1601, 1871, 1951, 2221, 2591, 3001, 3251, 3571, 3851, 4021, 4391, 4441, 4481, 4621, 4651, 4691, 4751, 4871.
\end{quote}
These values were reverse engineered by varying $h(x)$ in \eqref{eqn5Types12}.

For our analysis of $\mathbb Z_{27}$, we again divide the primes $q\equiv 1 \bmod 3$ into two sets:
\begin{align*}
\mathcal{T}^{(3)}_1 &= \{ q=N_1(\delta+3A(x-1)+9B)\; : \;A,B\in \mathbb Z,\; 3\nmid A\},\\
\mathcal{T}^{(3)}_2 &= \{ q=N_1(\delta+3A(x-1)+9B)\; : \; A,B\in \mathbb Z,\;3\mid A\},
\end{align*}
where $\delta =1$, $2$ or $4$ as $q\equiv 1$, $4$ or $7 \bmod 9.$
As above, we refer to these as {\em Type~1} or {\em Type~2} (with respect to the prime $3$).
We show in Lemmas~\ref{27normrep} and \ref{notboth} that $\mathcal{T}^{(3)}_1$ and $\mathcal{T}^{(3)}_2$ are disjoint and that every $q\equiv 1 \bmod 3$ lies in one of these two sets.

Again $\mathcal{T}^{(3)}_2$ is a known sequence, OEIS A014753, with several different characterizations.  For example, these are exactly the primes $q\equiv 1\bmod 3$ for which $4q=x^2+243y^2$ has a solution in integers, as can be seen by writing
$$ q=N_1(\delta +9a(x-1)+9B) \; \iff \; 4q=(2\delta +18B -27a)^2+243 a^2. $$
From \cite[ex.\ 9.23]{IrelandRosen} these are equivalently the $q\equiv 1\bmod 3$ for which $3$ is a cubic residue mod $q.$

We also need to further divide the Type~2 primes $q\equiv 1 \bmod 9$:
\begin{align*}
\mathcal{T}^{(3)}_3 &= \{ q=N_2(1+(x-1)^7h(x))\; : \;h\in \mathbb Z[x],\; 3\nmid h(1)\},\\
\mathcal{T}^{(3)}_4 &= \{ q=N_2(1+(x-1)^7h(x))\; : \; h\in \mathbb Z[x],\;3\mid h(1)\}.
\end{align*}
We refer to these as \textit{Type~3} and \textit{Type~4} (for the prime $3$).
We may now state our characterization of $\mathcal{S}(\mathbb Z_{27})$.

\begin{theorem}\label{thmZ27}
In addition to $3^6\mathbb Z$ and the integers coprime to $3$, the elements of $\S(\mathbb Z_{27})$ are as follows.
\begin{enumerate}[label=(\alph*)]
\item\label{thmZ27a}
The values with $3^4\parallel M_{27}(F)$ take the form $3^4mq$, with $3\nmid m$ and $q\equiv 1 \bmod 9$ a Type~1 prime.
\item\label{thmZ27b}
The values with $3^5\parallel M_{27}(F)$ take the form 
$3^5mq$, with $3\nmid m$, and $q\equiv 1 \bmod 3$ a Type~1 prime or $q\equiv 1 \bmod 9$ a Type~3 prime.
\end{enumerate}
\end{theorem}

For convenience, we list the elements of these sets less than $5000$.
We might reasonably expect to witness $3^4q$ for about $2/3$ of primes $q\equiv 1 \bmod 9$, and $3^5q$ for approximately $20/27$ of the primes $q\equiv 1 \bmod 3$.
Indeed, of the $109$ primes $q\equiv 1 \bmod 9$ less than this bound, we observe that $76$ are Type~1:
\begin{quote}
19, 37, 109, 127, 163, 181, 199, 379, 397, 433, 487, 541, 631, 739, 811, 829, 883, 937, 1063, 1153, 1171, 1279, 1297, 1423, 1459, 1567, 1657, 1693, 1747, 1801, 1873, 1999, 2017, 2053, 2089, 2143, 2161, 2377, 2467, 2503, 2521, 2539, 2557, 2593, 2647, 2683, 2719, 2791, 2917, 2953, 3061, 3169, 3259, 3313, 3331, 3457, 3511, 3547, 3583, 3637, 3673, 3691, 3709, 3727, 3943, 4051, 4159, 4231, 4357, 4447, 4519, 4591, 4663, 4789, 4861, 4987,
\end{quote}
while $24$ are Type~3:
\begin{quote}
73, 271, 307, 523, 577, 613, 757, 919, 1009, 1531, 1621, 1783, 2179,
2287, 2971, 3079, 3187, 3529, 3853, 3889, 4177, 4339, 4951, 4969
\end{quote}
and $9$ are Type~4:
\begin{quote}
991, 1117, 1549, 2251, 2269, 2341, 3907, 4483, 4933.
\end{quote}
Of the primes $q\equiv 4 \bmod 9$ up to $5000$, $75$ are Type~1:
\begin{quote}
13, 31, 139, 157, 211, 229, 283, 337, 373, 409, 463, 571, 607, 733, 
751, 769, 823, 859, 877, 1039, 1129, 1201, 1237, 1291, 1327, 1381, 
1453, 1471, 1489, 1579, 1723, 1741, 1777, 1831, 1993, 2011, 2083, 
2137, 2281, 2371, 2551, 2659, 2677, 2731, 2749, 2767, 2857, 3037, 
3163, 3271, 3307, 3361, 3433, 3469, 3541, 3559, 3613, 3793, 3919, 
4027, 4153, 4243, 4261, 4297, 4423, 4441, 4549, 4567, 4603, 4621, 
4639, 4657, 4801, 4909, 4999,
\end{quote}
and $34$ are Type~2:
\begin{quote}
67, 103, 193, 499, 643, 661, 787, 967, 1021, 1093, 1399, 1543, 1597, 
1669, 1759, 1867, 2029, 2389, 2713, 2803, 3001, 3019, 3109, 3181, 
3217, 3253, 3343, 3631, 3739, 3847, 4099, 4513, 4729, 4783.
\end{quote}
For the $q\equiv 7 \bmod 9$ in the same range, $76$ are Type~1:
\begin{quote}
7, 43, 79, 97, 223, 241, 277, 313, 331, 349, 421, 457, 601, 673, 
691, 709, 907, 1033, 1051, 1069, 1087, 1123, 1213, 1231, 1429, 1447, 
1483, 1627, 1663, 1699, 1753, 1789, 1951, 1987, 2113, 2239, 2293, 
2311, 2347, 2437, 2473, 2689, 2707, 2797, 2833, 2887, 3121, 3229, 
3301, 3391, 3463, 3571, 3643, 3697, 3769, 3823, 3877, 3931, 4003, 
4021, 4111, 4129, 4201, 4219, 4273, 4327, 4363, 4507, 4561, 4597, 
4651, 4723, 4759, 4813, 4831, 4903,
\end{quote}
and $36$ are Type~2:
\begin{quote}
61, 151, 367, 439, 547, 619, 727, 853, 997, 1249, 1303, 1321, 1609, 
1861, 1879, 1933, 2131, 2203, 2221, 2383, 2617, 2671, 2851, 3049, 
3067, 3319, 3373, 3499, 3517, 3607, 3733, 3967, 4057, 4093, 4957, 
4993.
\end{quote}

Last, in 1882 Torelli \cite{Torelli} showed that $\S(\mathbb Z_{rs}) \subseteq \S(\mathbb Z_s)$.
We describe an explicit construction (via Lemma~\ref{reduction} below) for a special case of this: writing a dimension $p^t$ integral circulant determinant as a dimension $p^{t-1}$ integral circulant determinant.
We record this in the following proposition.
(See also \cite[Lem.~3.6]{Norbert} for a similarly constructive proof.)

\begin{proposition}\label{nested}
For any $t\geq 2$ we have $\S(\mathbb Z_{p^t})\subseteq \S(\mathbb Z_{p^{t-1}}).$
\end{proposition} 

In this paper we are concerned with powers of odd primes.
When $p=2$ one has easily that
$$ \S(\mathbb Z_2)=\{2^am\; : \;2\nmid m,\;  a=0 \text{ or } a\geq 2\}. $$
Kaiblinger \cite{Norbert} showed that for $k=2$ or $3$ 
$$ \S(\mathbb Z_{2^k})=\{2^am\; : \;2\nmid m,\;  a=0 \text{ or } a\geq k+2\} $$
and obtained upper and lower bounds for $k\geq 4$,
$$ \{2^am\; : \;2\nmid m,\;  a=0 \text{ or } a\geq 2k-1\} \subseteq  \S(\mathbb Z_{2^k})\subseteq \{2^am\; : \;2\nmid m,\;  a=0 \text{ or } a\geq k+2\}. $$
Recently Y. and N.~Yamaguchi \cite{Yamaguchi3} resolved the case $\S (\mathbb Z_{16}),$ showing that the odd multiples of $2^6$ take the form $2^6p^2m$ with $p\equiv 3 \bmod 8$, or $2^6pm$ with $p\equiv 5 \bmod 8$,
or $2^6pm$  with $p\equiv 1 \bmod 8$ and $p=a^2+b^2$ with $a+b\equiv \pm 3 \bmod 8$.
 
This article is organized in the following way.
Section~\ref{secProof11} contains the proofs of Theorem~\ref{thmpowers} and Proposition~\ref{nested}.
Sections \ref{secPf25} and~\ref{secPf27} hold the proofs of Theorems \ref{thmZ25} and~\ref{thmZ27} respectively.
Section~\ref{secComputations} summarizes some computations associated with these results.

\section{Proof of Theorem~\ref{thmpowers}  and  Proposition~\ref{nested} }\label{secProof11}

We write $\Phi_n(x)$ for the $n$th cyclotomic polynomial and $\phi(n)$ for Euler's totient function.
We also let $\pi_i:=1-\omega_i$; recall that $(p)=(\pi_i)^{\phi(p^i)}$ gives the factorization of $(p)$ in $\mathbb Z[\omega_i].$ 

\begin{lemma}\label{Lemma1}
Let $G=\mathbb Z_{p^t}$ with $t\geq 2$.
If $p^j\parallel M_G(F)$ with $t+1\leq j\leq 2t$ then $p^{j-t}\parallel F(1)$ and $p\parallel N_i(F)$ for $1\leq i\leq t$.
\end{lemma}

\begin{proof}
Suppose that $M_G(F)=p^{j}m$, $p\nmid m$ with $j\geq 1$.
Writing
$$ N_j(F):=\prod_{\stackrel{\ell=1}{\gcd(\ell,p)=1}}^{p^j} F(\omega_j^{\ell}),\;\; \omega_j:=e^{2\pi i/p^j},\;\;\;
N_j(F) \equiv F(1)^{\phi(p^j)} \bmod p, $$
we must have $p\mid F(1),N_1(F),\ldots ,N_t(F)$.
Suppose that $p^u\parallel N_t(F)$.
Then
$$ F(x) =(1-x)^u g_1(x) + \Phi_{p^t} (x) g_2(x), \;\; g_1,g_2\in \mathbb Z[x],\;\; |g_1(\omega_t)|_p=1, $$
and for each $1\leq i\leq t-1$ we have
 $$ F(\omega_i) =(1-\omega_i)^ug_1(\omega_i) + pg_2(\omega_i),\;\; g_1(\omega_i)=g_1(\omega_t^{p^{t-i}})\equiv g_1(\omega_t)^{p^{t-i}} \bmod p. $$
Hence $ |g_1(\omega_i)|_p=1,$ and 
if $u=1$ we get $\pi_i\parallel F(\omega_i)$ and $p\parallel N_i(F)$, all $1\leq i\leq t$, while if $u\geq 2$ we get $\pi_i^2\mid F(\omega_i)$ and $p^2\mid N_i(F)$ all $1\leq i\leq t$ and $p^{2t+1}\mid M_G(F)$. 
\end{proof}

We next state Newman's unit criterion \cite[Thm.~1]{Newman2}, which we will use in the proofs of Theorems \ref{thmZ25} and \ref{thmZ27}.

\begin{lemma} \label{Newman}
If $u_1,u_2$ are units in $\mathbb Z[\omega_t]$ with $\pi_t^s \parallel  (u_1-u_2)$ with $s$ odd, then $s=p^\ell$ for some $0\leq \ell < t$.
\end{lemma}

Newman used his lemma to rule out $N_1(F)=N_2(F)=p$ when $F(1)=p.$ In our proof of Theorem~\ref{thmpowers} 
we will allow additional norm values that have a similar structure. Recall, if $q\neq p$ is a prime dividing $N_i(F)$, then its multiplicity in $N_i(F)$ must be a multiple of the order of $q \bmod p^i$.
(This follows by considering the factorization of primes in $\mathbb Z[\omega_i]$, see for example \cite[Thm.\ 2.13]{Washington}.)

\begin{lemma} \label{gettingreal}
Suppose that $F(x)$ in $\mathbb Z[x]$ has $p\parallel N_i(F)$ for some $i$, with 
\be \label{specialnorms} N_i(F)=p \;\;\; \text{  or  } \;\;\; N_i(F)=p\prod_{j=1}^{r} q_j^{f_{i,j}s_j},\;\; \text{$f_{i,j}$  all even}, \ee
where each $q_j$ is prime and $f_{i,j}$ is the minimal positive integer such that $q_j^{f_{i,j}}\equiv 1 \bmod p^i.$
Then for some $g(x)$ in $\mathbb Z[x]$ and integer $k$ we can write
$$ F(\omega_i) = \omega_i^k g(\omega_i+ \omega_i^{-1})(\omega_i-1). $$
\end{lemma}

\begin{proof} Since $|F(\omega_i)|_p<1$ we  have
$$F(\omega_i)=(\omega_i-1)h_1(\omega_i)+F(1),\;\; \;p\mid F(1),\;\; \; p=\prod_{\substack{1\leq j< p^i\\p\nmid j}}(\omega_i^j-1), $$
for some $h_1(x)\in\mathbb{Z}[x]$, so we can plainly write 
$$F(\omega_i)=(\omega_i-1)H(\omega_i),\;\;\; N_i(H)=1 \text{ or } N_i(H)=\prod_{j=1}^{r} q_j^{f_{i,j}s_j},$$
for some  $H(x)\in\mathbb Z[x]$.
In the first case $H(\omega_i)$ is a unit and we can take a basis for the units in  $\mathbb Z[\omega_i]$ that is polynomial in $\omega_i+\omega_i^{-1}$: we know (see for example \cite[Cor.~4.13]{Washington}) that if $u(\omega_i)$ is a unit then $\overline{u(\omega_i)}=\omega_i^cu(\omega_i)$ for some integer $0\leq c<p^i,$ so if we replace 
$u(x)$ by $v(x)=x^ku(x)$ with $2k\equiv c \bmod p^i$, then $\overline{v(\omega_i)}=v(\omega_i)$ and $v(\omega_i)$ is in the  maximal real subfield whose ring of integers is $\mathbb Z[\omega_i+\omega_i^{-1}]$ (see \cite[Prop.~2.16]{Washington}).

In the second case we can obtain something similar.
Observe that if $f_{i,j}$ is even then $q_j^{f_{i,j}/2}\equiv \pm 1 \bmod p^i$ and (e.g., \cite[ex.~4.12]{Marcus}) the prime ideal
$(q_j)$ breaks into the same number of prime ideals in $\mathbb Z[\omega_i]$ as it does in  the maximal real subfield $\mathbb Z[\omega_i+\omega_i^{-1}].$ In particular we can assume those primes in $\mathbb Z[\omega_i]$ have the same
real generators and are unchanged by complex conjugation.
Hence if all the $f_{i,j}$ are even we have $\left(H(\omega_i)\right)=(\overline{H(\omega_i)})$ in $\mathbb Z[\omega_i]$ and $H(\omega_i)/\overline{H(\omega_i)}$ is an algebraic integer all of whose conjugates have absolute value~$1$.
From Kronecker's Theorem \cite{Kronecker} this must be a root of unity and $H(\omega_i)=\pm \omega_i^c \overline{H(\omega_i)}.$ Further we can assume that $H(\omega_i)= \omega_i^c \overline{H(\omega_i)},$ since $H(\omega_i)=-\omega_i^c\overline{H(\omega_i)}\equiv -H(1) \equiv -H(\omega_i) \bmod (1-\omega_i)$
contradicts $|2H(\omega_i)|_p=1.$ So, writing $\omega_i^c=\omega_i^{2k},$ we have that $\overline{\omega_i^{-k}H(\omega_i)}=\omega_i^{-k}H(\omega_i)$ is real and $H(\omega_i)=\omega_i^{k}g(\omega_i+\omega_i^{-1})$ for some $g(x)\in\mathbb Z[x].$
\end{proof}

\begin{proof}[Proof of Theorem~\ref{thmpowers}]
From Newman \cite[Thm.~4]{Newman1} we know that $p^{2t}\mathbb Z\subset \S(\mathbb Z_{p^t});$ to be explicit $M_{p^t}\left( x-1+ m\prod_{j=1}^t \Phi_{p^j}(x)\right)=mp^{2t}$  for any $m.$

Suppose that $p\geq 5$ and $t\geq 2$ or $p=3$ and $t\geq 3$ and that we have a value $M_{p^t}(F)=mp^{j}$ with  $p\nmid m$ and $t+1\leq j \leq 2t-1.$ Then by Lemma \ref{Lemma1} we must have $F(1)=m_1p^{j-t},$  some $m_1\mid m,$ 
and $p\parallel N_i(F)$ for all $1\leq i\leq t.$ Hence if 
fewer than $2t-j$  of the $N_i(F)$ are divisible by a $q_i^{f_i}\equiv 1 \bmod p^i$ with $f_i$ odd  (fewer than $t-2$ when $p=3$ and $j=t+1$) we must have at least $j-t+1$ of the $N_i(F)$ of the form \eqref{specialnorms} (with $i\geq 2$ when $p=3$, $j=t+1$). Hence, setting $k=j-t,$
we have $F(1)=m_1p^k$ and $N_{j_i}(F)$ of the form \eqref{specialnorms} for $k+1$ norms $1\leq j_k<\cdots <j_1< j_0\leq t$ 
(with $2\leq j_1<j_0\leq t$ when $p=3$ and $j=t+1$) where $1\leq k\leq t-1$.

Hence for $i=j_0$ we have  $F(\omega_{i})=(\omega_{i}-1)\omega_i^Ig(\omega_i+\omega^{-1}_i),$ 
where we can take $I \bmod p^i$ large enough that $x^Ig(x+1/x)\in \mathbb Z[x].$ So, from the value at 
$x=\omega_{j_0}$ and $x=1,$ there is some $h(x)$ in $\mathbb Z[x]$ such that
$$ F(x)=(x-1)x^Ig(x+1/x)+ m_1\prod_{i=0}^{k-1} \Phi_{p^{j_i}}(x)  + (x-1)\Phi_{p^{j_0}}(x)h(x). $$

For $\ell=1,\ldots,k,$ we set  $\lambda_{\ell}:=N_{j_{\ell}}(F)/\omega_{j_{\ell}}^I(\omega_{j_{\ell}}-1).$ Then for $x=\omega_{j_{\ell}}$ and ${I'}=-I \bmod p^t$ we can write
$$ \lambda_{\ell} = g(x+1/x) + p x^{I'} h(x) = g(x+1/x) + p \sum_{0\leq j\leq S} a_j (x-1)^j $$
for $\ell=1,\ldots,k-1,$ and for $\ell=k$ except we must add the term
$$ m_1x^{I'}p^k/(x-1)= p^{k-1} (x-1)^{\phi(p^{j_{k}})-1}(A+(x-1)s_1(x)), \;\; p\nmid A. $$
As in the proof of Newman's Lemma we use Lemma~\ref{gettingreal} to take conjugates and write $\overline{\lambda_{\ell}}=\omega_{j_l}^{c_{\ell}} \lambda_{\ell}$
for some $0\leq c_{\ell}  < p^{j_{\ell}}.$ Subtracting and  expanding the resulting polynomial in powers of $(x-1)$ gives
\begin{align}  (1-\omega_{j_{\ell}}^{c_{\ell}}) \lambda_{\ell} =\lambda_{\ell}-\overline{\lambda_{\ell}} & =p \sum_{\stackrel{1\leq i\leq S}{i \text{ odd}}} a_i(x-1)^i(1+x^{p^t-i})+ p \sum_{\stackrel{1\leq i\leq S}{i \text{ even}}} a_i(x-1)^i(1-x^{p^t-i})\nonumber \\
 & \label{everything} = \sum_{1\leq i\leq S'} pA_i (x-1)^i, \;\;\; A_i\in \mathbb Z,
\end{align}
for $\ell=1,\ldots,k-1.$ For $\ell=k$ we have an additional term
\be \label{extra}  p^{k-1}(x-1)^{\phi(p^{j_{k}})-1}(2A+(x-1)s_2(x)). \ee
Notice that the left side of \eqref{everything} is zero if $p^{j_{\ell}}\mid c_{\ell}$ and exactly divisible by $\pi_{j_{\ell}}=1-\omega_{j_{\ell}}$ to the power $\gcd(c_{\ell},p^{j_{\ell}-1})$ otherwise (see for example \cite[Lem.~1]{Newman2}). 

Suppose first that $j=t+1$, $k=\ell=1.$ Then $p$ and hence $\phi(p^{j_1})$ many $\pi_{j_{1}}$ divide the terms in \eqref{everything}
while  exactly $\phi(p^{j_1})-1$ many $\pi_{j_1}$ divide the extra term \eqref{extra}. This rules out the left-hand side being zero, moreover $\phi(p^{j_1})-1=p^{j_{\ell}-1}(p-1)-1$
is plainly not a multiple of $p$ and is greater than $1$ (since we ruled out $j_1=1$ when $p=3$) and so cannot equal $\gcd(c_{\ell},p^{j_{\ell}-1})$.

Now suppose that $k\geq 2.$ The terms in  \eqref{everything}  and \eqref{extra} are all multiples of $p,$ so divisible by at least $\phi(p^{j_{\ell}})>p^{j_{\ell}-1}$ many $\pi_{j_{\ell}}.$ Hence the left-hand side  must be zero and we will
get a contradiction if the right-hand side is only divisible by a finite power of $\pi_{j_{\ell}}.$

We claim that the $\ell=1,\ldots,k-1$  successively give 
\be \label{successive}  p^{\ell} \mid A_i,\;\; 1\leq i \leq  \phi(p^{j_{\ell}}),\ee
Observe that if $p^{m_i} \parallel A_i$ then $\pi_{j_{\ell}}^{t_i}\parallel  pA_i (\omega_{j_{\ell}}-1)^i $ with $t_i=(1+m_i)\phi(p^{j_{\ell}})+i$ and the $t_i$ are distinct for $1\leq i \leq \phi(p^{j_{\ell}})$. 

Hence for $\ell=1$, if $p\nmid A_J$ for some  $1\leq J\leq \phi (p^{j_{1}})$, $J$ minimal, 
we have $t_J=\phi(p^{j_{1}})+J$ while the other $t_i$ have $i>J$ or $m_i\geq 1$ and $t_i>t_J.$ Hence we get that $t_J$ is the exact power of $\pi_{j_1}$ dividing the right-hand side and  a contradiction. 
So $p\mid A_i$ for all $1\leq i \leq \phi(p^{j_1})$.

Suppose that we have shown \eqref{successive} for $\ell=1,\ldots,L-1,$  $L\geq 2,$  and want to show it for $L$. We certainly know that $p^{L-1}\mid A_i$
all $1\leq i\leq \phi(p^{j_{L}})$, since $j_L<j_{L-1},$ so suppose that $p^{L-1}\parallel A_J$ for some $1\leq J\leq \phi(p^{j_L})$ with $J$ minimal, then plainly $t_J=L\phi(p^{j_L})+J<t_i$ for the other $t_i$ with $1\leq i\leq \phi(p^{j_{L-1}})$ since either $m_i\geq L$ or $m_i=L-1$ and $i>J$. 
If $\phi(p^{j_{L-s}})< i \leq\phi( p^{j_{L-s-1}})$ for some $1\leq s\leq L-2$ then $m_i\geq {L-s-1}$ and
$t_i> (L-s)\phi(p^{j_L})+ \phi(p^{j_{L-s}}) \geq (L-s+p^s)\phi(p^{j_L})\geq (L+1)\phi(p^{j_L})\geq t_J.$
If $i>\phi(p^{j_1})$ then $t_i> \phi(p^{j_1})\geq  \phi(p^{j_L +L-1})= \phi(p^{j_L})p^{L-1}\geq (L+1)\phi(p^{j_L})\geq t_J.$ Again a contradiction and the claim holds for $\ell=L.$

Hence for $\ell = k$ we have $p^{k-1} \mid A_i$ for all $1\leq i\leq \phi (p^{j_{k-1}})$ and 
$t_i\geq k \phi(p^{j_k}).$ For the $\phi(p^{j_{k-s}}) < i\leq \phi(p^{j_{k-s-1}})$ for some $1\leq s\leq k-2$ we
have $t_i> (k-s)\phi(p^{j_k})+ \phi(p^{j_{k-s}})\geq (k-s+p^s)\phi(p^{j_k})\geq k\phi(p^{j_k})$ and for the $i>\phi(p^{j_1})$
we have $t_i\geq i > \phi(p^{j_1})\geq p^{k-1} \phi(p^{j_k})\geq k\phi(p^{j_k}). $ But for $\ell=k$ we have an extra term
\eqref{extra} and the power of $\pi_{j_k}$ dividing that is exactly  $k\phi(p^{j_k}) -1$. Hence this is the power of $\pi_{j_k}$ dividing the right-hand side,  contradicting that it is zero. \end{proof}

Finally, for the proofs of Theorems \ref{thmZ25} and \ref{thmZ27} we shall need a lemma writing a $p^j$ norm as a $p^{j-1}$ norm.
Proposition~\ref{nested} then follows immediately.

\begin{lemma} \label{reduction}  If $F(x)$ is in $\mathbb Z[x]$ and
$ H(x)=\prod_{j=0}^{p-1} F(x\omega_1^j), $ then $H(x)=g(x^p)$ with
$g(x)$ in $\mathbb Z[x],$  $g(x)\equiv f(x) \bmod p$ and
$$N_{i}(F)=N_{i-1}(g),\;\; i\geq 2, \hspace{3ex} F(1)N_1(F)=g(1).$$
If  $p=3$ and  $F(x)=f_0(x^3) + x f_1(x^3) + x^2 f_2(x^3)$ then
\be \label{defggg}  g(x)= f_0(x)^3 + xf_1(x)^3+ x^2f_2(x)^3-3xf_0(x)f_1(x)f_2(x). \ee
\end{lemma}
\begin{proof}
To see this, observe that $\prod_{j=0}^{p-1} F(\omega_1^jx)$ is symmetric in the conjugates so will be in $\mathbb Z[x]$ and unchanged under $x\mapsto \omega_1 x$ so must be a polynomial in $x^p.$ Now $H(x)\equiv F(x)^p \bmod (1-\omega_1)$ and since the coefficients are integers $g(x^p)\equiv F(x)^p\equiv F(x^p) \bmod p.$
 For $i\geq 2$ the $x\omega_1^j$, $j=0$, \ldots, $p-1$ will run through
the primitive $p^i$th roots of unity as $x$ runs through the $\omega_i^{\ell}$, $1\leq \ell<p^{i-1}$, $p\nmid \ell$, i.e.,  the $x^p$ in $g(x^p)$ runs  through the primitive $p^{i-1}$th roots of unity. Hence $N_i(F)=N_{i-1}(g)$.
\end{proof}

\begin{proof}[Proof of Proposition~\ref{nested}]
Immediate  from Lemma \ref{reduction} since $M_{p^i}(F)=M_{p^{i-1}}(g).$ 
\end{proof}

As shown in Kaiblinger \cite[Lem.\ 3.6]{Norbert} we can equivalently think of $g(x)$ as the polynomial whose roots and lead coefficient are the $p$th power of those of $F(x).$
Writing group determinants as subgroup determinants is explored in \cite{Yamaguchi1,Yamaguchi2}.

\section{Proof of Theorem~\ref{thmZ25}}\label{secPf25}

Throughout this section, $\omega_i$ denotes a primitive $5^i$th root of unity.
Note that $\mathbb Z[\omega_i]$ is a unique factorization domain for $i=1$ and $2$ (see \cite[Thm.~11.1]{Washington}).
We begin with a way of using units to reduce elements in  $\mathbb Z[\omega_i].$

\begin{lemma}\label{lemNormRep}
Suppose that $f(x)\in \mathbb Z[x]$ with $5\nmid N_i(f)$.
\begin{enumerate}[label=(\roman*)]
\item \label{lemNormRep1} For some unit $u(\omega_i)$ and $h(x)\in \mathbb Z[x]$ we can write
\begin{equation}\label{genred}
u(\omega_i) f(\omega_{i}) =1+ (\omega_{i}-1)^3 h(\omega_{i}).
\end{equation}
We can take $u(x)=\pm x^I(1+x)^J(x^4+x)^{2K}$,   with $J=0$ or $1$ and $0\leq I,K\leq 4.$
\item \label{lemNormRep2}
$f(\omega_i)$ cannot have one representation \eqref{genred} where $5\mid h(1)$ and another where $5\nmid h(1).$
\end{enumerate}
\end{lemma}

\begin{proof}
\ref{lemNormRep1} Writing $f(\omega_i) =a_0+\sum_{j\geq 1} a_j(\omega_i-1)^2$, note that $5\nmid a_0$ since $5\nmid N_i(f)$. Observing that
$(1+\omega_i)=2+(\omega_i-1)$ and $5=(\omega_i-1)^{\phi(5^i)}v(\omega_i)$ then, taking $J=0$ or $1$ as $a_0\equiv \pm 1$ or $\pm 2 \bmod 5$, we can write
$$ \pm (1+\omega_i)^Jf(\omega_i) = 1+ b_1(\omega_i-1) +\sum_{j\geq 2} b_j(\omega_i-1)^j. $$
Since 
$\omega_i^I=1+I(\omega_i-1) + \sum_{j\geq 2} t_j(\omega_i-1)^j$, taking $I\equiv -b_1 \bmod 5$ we get
 $$ \pm \omega_i^I(1+\omega_i)^Jf(\omega_i) = 1+ c_2(\omega_i-1)^2+\sum_{j\geq 3} c_j(\omega_i-1)^j. $$
Using 
\begin{align*}
-(\omega_i^4+\omega_i)^2 &= -4-20(\omega_i-1) -49(\omega_i-1)^2-76(\omega_i-1)^3-\cdots\\
&= 1+(\omega_i-1)^2+\sum_{j\geq 3} r_j(\omega_i-1)^j,
\end{align*}
and $K\equiv -c_2 \bmod 5$ we get
\[
\pm \omega_i^I(1+\omega_i)^J(-(\omega_i^4+\omega_i)^2)^Kf(\omega_i) = 1+\sum_{j\geq 3} d_j(\omega_i-1)^j. 
\]

\ref{lemNormRep2} Suppose that
\begin{equation*}\label{bothtypes}
\begin{split}
u_1(\omega_i) f(\omega_i) & =1+(\omega_i-1)^3(A+(\omega_i-1)g_1(\omega_i)), \;\; 5\mid A,\\
u_2(\omega_i) f(\omega_i)&=1+(\omega_i-1)^3(B+(\omega_i-1)g_2(\omega_i)), \;\; 5\nmid B, 
\end{split}
\end{equation*}
for two units $u_1(\omega_i)$, $u_2(\omega_i)$.
Then
$$ u_1(\omega_i)-u_2(\omega_i) = (\omega_i-1)^3\left( u_2(\omega_i) A-u_1(\omega_i)B+ (1-\omega_i)h(\omega_i)\right) $$
and $ (1-\omega_i)^3 \parallel u_1(\omega_i)-u_2(\omega_i), $
contradicting Lemma~\ref{Newman}. 
\end{proof}

We say that an integer  $m$ is a \textit{$5^i$-norm} if $m=N_i(f(x))$ for some $f(x)\in \mathbb Z[x].$

\begin{lemma} \label{primenorm}
If $5\nmid m$ and $m$ is a $5^i$-norm, then
\be \label{5form1} m=N_i(1+(x-1)^3h(x)) \ee
 for some $h(x)\in \mathbb Z[x]$. For $i=1$ we can also write
\be \label{5form} 5m = N_1(1-x +5g(x)), \ee
for some $g(x)\in \mathbb Z[x],$  with $5\nmid g(1)$ if $5\nmid h(1)$.
\end{lemma}

\begin{proof}
The form \eqref{5form1} follows at once from Lemma~\ref{lemNormRep}.
The second form \eqref{5form} stems from the observation that 
$5=N_1(1-x)$ and that $(x-1)^4=5(x^2+1)^2-4\Phi_5(x).$
\end{proof}

We call \eqref{5form1} a \textit{$5^i$-norm Type~1 representation} for $m$ if $5\nmid h(1),$ and a \textit{$5^i$-norm Type~2 representation} if $5\mid h(1)$. Crucial will be integers with a $5$- or $25$-norm Type~1  representation
(though in fact the $25$-norm gives us no new values).

\begin{lemma} \label{25normsare5norms} If $m$ has a $25$-norm Type~1 (or Type~2) representation, then 
$m$ has a $5$-norm Type~1 (or Type~2) representation.
\end{lemma}

\begin{proof} Suppose that $m=N_2(F)$ with $F(x)=1+(x-1)^3h(x)$.
Then by Lemma~\ref{reduction} we have
$m=N_1(g)$ with $g(x)=F(x)+5h_1(x)$ and $m=N_1(1+(x-1)^3h(x)+ (x-1)^4(x^3-x-1)h_1(x)).$  This gives a 5-norm representation which is Type~1 if $5\nmid h(1)$ and Type~2 if $5\mid h(1)$.
\end{proof}

We observe that we can achieve $5^3m$ when $m$ has a $5$-norm Type~1 representation, for example, the Type~1 primes $q\equiv 1 \bmod 5$.

\begin{lemma} \label{repping}
Let $m\in\mathbb{Z}$.
\begin{enumerate}[label=(\roman*)]
\item\label{repping1}
If $m$ has a $5$-norm Type~1 representation then there is an $F(x)$ in $\mathbb Z[x]$ with 
$M_{25}(F)=5^3m$ and $F(1)=N_2(F)=5$, $N_1(F)=5m.$
\item\label{repping2}
If $m$ has a $25$-norm Type~1 representation then there is additionally an $F(x)$ in $\mathbb Z[x]$ with 
$M_{25}(F)=5^3m$ and  $F(1)=N_1(F)=5$, $N_2(F)=5m.$
\end{enumerate}
\end{lemma}

\begin{proof}
\ref{repping1} Suppose that $5m$ has a form \eqref{5form} with $5\nmid g(1)$. So there exists a positive integer
$\ell$ with $5\nmid \ell$ and $\ell g(1)=1+5\lambda$ for some integer $\lambda$. Take 
\be   \label{Fwithf}    F(x)=(1+x+\cdots +x^{\ell-1})f(x) -\lambda \left(\frac{x^{25}-1}{x-1}\right),    \ee
with $f(x)= (1-x) + \Phi_5(x^5) g(x).$ Then $F(1)=5\ell g(1)-25\lambda=5$, $N_1(F)=N_1(f)=5m$, $N_{2}(F)=N_2(f)=N_{2}(1-x)=5$ and $M_{25}(F)=5^3m.$

\vspace{1ex}
\noindent
\ref{repping2} If $m=N_2\left(1+(x-1)^3h(x)\right),$ $5\nmid h(1),$ then
$$ (\omega_1-1)\left(1+(\omega_1-1)^3h(\omega_1)\right)=(\omega_1 -1)+ 5h(\omega_1)(\omega_1^2+1)^2 $$
and taking integers  $\ell>0$ and $\lambda$ with $-4h(1)\ell=1+5\lambda$ we set \eqref{Fwithf} with
$$ f(x)=(x-1)\left(1+(x-1)^3h(x)\right)- \Phi_5(x^5) h(x)(x^2+1)^2.  $$
So $F(1)=-  20h(1)\ell-25\lambda =5,$ $N_1(F)=N_1(f)=N_1(x-1)=5,$ $N_2(F)=N_2(f)=5m$ and $M_{25}(F)=5^3m.$
\end{proof}

For a prime $q$ we use $q^f$ to denote the minimal positive power of $q$ with $q^f\equiv1 \bmod 5^i$.
Composite $m$ can have both Type~1 and Type~2 representations but this is not true for the $q^f$.
Thus it makes sense to call $q^f$ a $5^i$-norm Type~1 or Type~2 integer depending on whether it has a  $5^i$-norm Type~1 or~2 representation \eqref{5form1}.

\begin{lemma}\label{Type1or2}
A minimal prime power $q^f\equiv 1 \bmod 5^i$ cannot be both $5^i$-norm Type~1 and Type~2.
If $m$ is a product of  $q^{f}$ all  of $5^i$-norm Type~2 or $m=1$ then $m$ can only be  $5^i$-norm Type~2.
A product of one Type~1 and the rest Type~2 must be Type~1. If at least two  are Type~1 then $m$ will have both Type~1 and Type~2 representations. 
\end{lemma}

\begin{proof}
Suppose that $g(x)$ in $\mathbb Z[x]$ has $N_i(g)=q^f$ and that
\[
u_1(\omega_i)g(\omega_i)=1+(\omega_i-1)^3h_1(\omega_i)
\]
for some unit $u_1(\omega_i)$.
Then for any $5\nmid j$ we have 
$$ u_1(\omega_i^j) g(\omega_i^j)=1+(\omega_i-1)^3h_2(\omega_i),\;\;\; h_2(x)=(1+x+\cdots + x^{j-1})^3h_1(x^j).$$ 
So if we have a representation in \eqref{5form1} with $5\mid h_1(1)$ or $5\nmid h_1(1)$ then all the conjugates 
will also have a representation of that same type.
Hence $q^f$ cannot have  both a Type~1 representation $g_1(\omega_i)$ and a Type~2 representation $g_2(\omega_i);$  up to a unit $g_1(\omega_i)$ is a conjugate of $g_2(\omega_i)$, and Lemma \ref{lemNormRep}\ref{lemNormRep2} rules out $g_1(\omega_i)$ having both types of representation.

Observe that if $f(\omega_i)$ has a Type~2 representation and  $g(\omega_i)$  has a Type~1 or~2, then $f(\omega_i)g(\omega_i)$ has a representation of the same type as $g(\omega_i)$:
\begin{align*}  u_1(\omega_i) f(\omega_i) & = 1+(\omega_i-1)^3(A+ (1-\omega_i)h_1(\omega_i)), \;\;5\mid A,   \\
u_2(\omega_i) g(\omega_i)& =1+ (\omega_i-1)^3(B+ (1-\omega_i)h_2(\omega_i)),
\end{align*}
then 
$$ u_1(\omega_i)u_2(\omega_i)f(\omega_i)g(\omega_i) =  1+\left(\omega_i-1)^3(A+B+ (1-\omega_i)h_3(\omega_i))\right)$$
and $5\mid A+B$ if and only if $5\mid B$. If $5\nmid AB$ then we can choose $j$ with $5\nmid j$ so that $5\nmid jA+B$ or so that $5\mid jA+B$ and $f(\omega_i^j)g(\omega_i)$ is  respectively  Type~1 or Type~2.

Now if $m$ is a product of Type~2 prime powers then the constituent primes will all have Type~2 representations and their product a Type~2 representation.
A unit $u(\omega_i)$ plainly has a Type~2 representation $u(\omega_i)^{-1}u(\omega_i)=1.$
As shown in Lemma~\ref{lemNormRep}\ref{lemNormRep2}, an element in $\mathbb Z[\omega_i]$ cannot have both types of representation. The other cases are similar.
\end{proof}

We show that the only minimal $q^f$ of 5-norm Type~1 are the Type~1 primes $q\equiv 1 \bmod 5$
and the only $25$-norm Type~1 are the Type~1 primes $q\equiv 1 \bmod 25$.

\begin{lemma}\label{primepowers}
Let $q$ be a prime.
\begin{enumerate}[label=(\roman*)]
\item\label{primepowers1} If $q\equiv \pm 2 \bmod 5$ then $q^4$ is $5$-norm Type~2.
\item\label{primepowers2} If $q\equiv -1 \bmod 5$ then $q^2$ is $5$-norm  Type~2.
\item\label{primepowers3} The only $25$-norm Type~1 minimal $q^f$ are the Type~1 primes $q\equiv 1 \bmod 25$.
\end{enumerate}
\end{lemma}

\begin{proof}
\ref{primepowers1} Observe that $5\omega_1^3+5\omega_1^2-3$ is a unit in $\mathbb Z[\omega_1].$
Hence $3+5k$ has 
$$ (3+5k)(5\omega_1^3+5\omega_1^2-3)=1-(\omega_1-1)^4\left(5\omega_1^3-8\omega_1-8+k(8\omega_1^3-13\omega_1-13)\right), $$
 and a Type~2 representation, giving the result for $q^4$ when $\pm q\equiv 3 \bmod 5$.

\vspace{1ex}
\noindent
\ref{primepowers2} If $q\equiv -1 \bmod 5$, then $\left(\frac{5}{q}\right)=1$ and
$q$ splits in the real subfield $\mathbb Q[\sqrt{5}]$, and $q^2=N_1(f(\sqrt{5}))$ for a linear function $f(x)\in\mathbb{Q}[x]$.
Hence for some unit $u(\omega_1)$ we have
$$ u(\omega_1)f(\sqrt{5})= 1+(\omega_1-1)^3 h(\omega_1),\;\;\; \overline{u(\omega_1)}f(\sqrt{5})= 1-(\omega_1-1)^3 \omega_1^2\:\overline{h(\omega_1)}. $$
Hence $u(\omega_1)-\overline{u(\omega_1)}=(\omega_1-1)^3k(\omega_1)$ with $k(\omega_1)=u(\omega_1)\omega_1^2h(\omega_1^{-1})+u(\omega_1^{-1})h(\omega_1)\equiv 2u(1)h(1) \bmod (\omega_1-1)$ and $5\nmid h(1)$ would give $(\omega_1-1)^3\parallel (u(\omega_1)-\overline{u(\omega_1)})$, contradicting Lemma \ref{Newman}. 

\vspace{1ex}
\noindent
\ref{primepowers3} Suppose that $q^f$  is $25$-norm Type~1.
By Lemma~\ref{25normsare5norms} we know that $q^f,$ and hence the minimal power $q^{f'}\equiv 1 \bmod 5$, must be $5$-norm Type~1.
By \ref{primepowers1} and \ref{primepowers2} this says that  $q\equiv 1 \bmod 5$ is a Type~1 prime.
If $q\equiv 1 \bmod 25$ then $f=f'=1$ and  $q$ will have the same $25$-norm and $5$-norm type. If $q\equiv 1 \bmod 5$ but $q\not\equiv 1 \bmod 25$, then $f=5$, and taking a representation $q=N_1(1+(x-1)^3h(x))$ of either type gives a $25$-norm Type~2 representation for $q^5$:
\begin{align*}
q^5 & =N_2(1+(x^5-1)^3h_1(x^5))\\
&=N_2(1+(x-1)^{15}h_1(x^5)+5h_2(x))\\
& =N_2(1+(x-1)^{15}h(x^5)+(x-1)^{20}h_3(x)). \qedhere
\end{align*}
\end{proof}

Tanner \cite{Tanner1,Tanner2} defined a prime $q\equiv 1 \bmod 5$ to be an \textit{artiad} if the coefficients in the Jacobi function
$$ R(\omega_1)=\sum_{s=1}^{q-2} \chi(s)\chi(s+1),\;\; \chi(s)=\omega_1^{\ind(s)}, $$
the index $\ind(s)$ taken relative to a primitive root mod $q$, normalized so that 
\be \label{Rexp} R(\omega_1)=\sum_{i=0}^4 q_i \omega_1^i,\;\; \;\;\sum_{i=0}^4 q_i=-1, \ee
satisfy
$$ q_1\equiv q_2\equiv q_3\equiv q_4 \bmod 5.$$
Otherwise he defined it to be a \textit{perissad}.
We show that  this  coincides with  our partition \eqref{eqn5Types12}.

\begin{lemma}\label{artiad}
A prime $q\equiv 1\bmod 5$ is an artiad if and only if $q \in \mathcal{T}^{(5)}_2.$
\end{lemma}

\begin{proof}
From $R(\omega_1)R(\omega^{-1})=q$ we can write $R(\omega_1)=u_1(\omega_1)f(\omega_1)f(\omega_1^2)$ where $q=N_1(f).$
Now if $u_2(\omega_1)f(\omega_1)=1+(\omega_1-1)^3h(\omega_1)$ then 
\begin{align*} u_1(\omega_1)^{-1}u_2(\omega_1)u_2(\omega_1^2) R(\omega_1) & =\left(1+(\omega_1-1)^3h(\omega_1)\right)\left(1+(\omega_1^2-1)^3h(\omega_1^2)\right)\\
 &= 1+(\omega_1-1)^3\left(9h(1)+(\omega_1-1) h_1(\omega_1)\right)
\end{align*}
and $R(\omega_1)$ has a Type 2 representation if and only if $q$ is a Type 2 prime.

Rewriting \eqref{Rexp} and using \cite[eq.\ (16)]{ELehmer} we get
\begin{align*}
R(\omega_1) &= -1 + q_1(\omega_1-1) + q_2(\omega_1^2-1)+q_3(\omega_1^3-1) + q_4(\omega_1^4-1)\\
&= -1-5q_1+(\omega_1^2-1)(2q_4-2q_1+q_2-q_3)+(\omega_1^2-1)^2(q_3+q_4-2q_1)  \\
&\qquad -\omega_1^3(\omega_1^2-1)^3(q_3-q_1) \\
&= -1 -\omega_1^3(\omega_1^2-1)^3(q_3-q_1)-5q_1-5v(\omega_1^2-1)\\
&\qquad +5(u+2w+q_3-q_1)(\omega_1^2-1)^2
\end{align*}
with $u$, $v$ and $w$ as in \eqref{Dicksonform}, which gives a Type~2 representation if and only if $q_1\equiv q_3\bmod 5.$ 
Further if $q_1\equiv q_3\bmod 5,$ then the third equation in \cite[eq.\ (16)]{ELehmer} gives $q_2\equiv q_4\bmod 5$ and the first one yields $q_1\equiv q_2\bmod 5.$ So $q$ is Type~2 if and only $q_1\equiv q_2\equiv q_3\equiv q_4\bmod 5.$
\end{proof}

\begin{proof}[Proof of Theorem~\ref{thmZ25}]
We know from Lemma~\ref{repping} that we can achieve $5^3q$ for any Type~1 prime $q\equiv 1 \bmod 5$.
From the classic results of Newman and Lacquer we can achieve any $m$ with $\gcd(m,5)=1,$ and so by multiplicativity the $5^3mq;$ indeed we just need to multiply the polynomial by $\prod_{q^{\alpha}\parallel m} \Phi_q(x)^{\alpha},$ changing $F(1)$ by a factor of $m$ and leaving the other $N_i(F)$ unchanged.

It remains to show that  any  $M_G(F)=5^3m$, $5\nmid m,$  must be divisible by a Type~1 prime $q\equiv 1 \bmod 5$.
We show slightly more; namely that we cannot have $F(1)=5m_0$, $5\nmid m_0,$ $N_1(F)=5m_1,$ $N_2(F)=5m_2$ with $m_1$ only 5-norm Type~2 and $m_2$ only $25$-norm Type~2.
In particular, by Lemma~\ref{Type1or2} and Lemma~\ref{primepowers}, this says that  $N_1(F)$ must be divisible by a Type~1 prime $q\equiv 1 \bmod 5$ or $N_2(F)$ by a Type~1 prime $q\equiv 1 \bmod 25$. 

Suppose then that $N_2(F)=5m_2$ where $m_2$  is only $25$-norm Type~2.
Setting $F(\omega_2)=(1-\omega_2)f(\omega_2)$, we can apply Lemma~\ref{lemNormRep} to $f(\omega_2)$ and use the 
resulting $25$-norm Type~2 representation for $m_3$ and then the value $F(1)=5m_0$ to write
$$ F(x) = (1-x)u_1(x) \left(1+ (x-1)^3 h(x)\right) +\Phi_5(x^5)g_1(x),\;\;\; u_1(x)=\prod_{\stackrel{2\leq  j<25}{5\nmid j}} u(x^j) $$
with $5\mid h(1)$ and $u(x)=\pm x^I (1+x)^J(x^4+x)^{2K},$ for some $I,J,K\geq 0$, and
$$ g_1(x)= m_0 + (1-x)g(x), $$
for some $g\in \mathbb Z[x].$
Hence
$$ N_1(F)=5 N_1\left(1+(x-1)^3h(x) +\left(\prod_{j=2}^4u_1(x^j)(1-x^j)\right)(m_0+ (1-x)g(x))\right).$$
But $5\nmid m_0$ and $5\mid h(1)$ results in a 5-norm Type~1 representation for $m_1$.
\end{proof}

\section{Proof of Theorem \ref{thmZ27}}\label{secPf27}

Throughout this section, $\omega_i$ denotes a primitive $3^i$th root of unity.
Note that $\mathbb Z[\omega_i]$ is a unique factorization domain for $i=1$, $2$ and $3$ \cite[Thm.~11.1]{Washington}.
We begin again with a way of using units to reduce elements in $\mathbb Z[\omega_i].$

\begin{lemma}\label{27normrep}
Suppose that $f(x)\in \mathbb Z[x]$ with $3\nmid N_1(f)$.
\begin{enumerate}[label=(\roman*)]
\item\label{27normrep1}
For some  $J\in\{0,1,2\}$ and $\delta =1$, $2$ or $4$ as $N_1(f)\equiv 1$, $4$ or $7 \bmod 9$, we have
\be \label{3normforma}  \pm \omega_1^J f(\omega_1) =\delta + 3A(\omega_1-1) +9b, \;\; A,b\in \mathbb Z, \ee
and
\be \label{3normformb}  \pm \omega_1^J f(\omega_1) (1-\omega_1)=\delta (1-\omega_1) +9A +9B(1-\omega_1), \;\; B=b-A. \ee
\item\label{27normrep2}
For any $i\geq 2$ 
we can write
\be \label{genred3} u(\omega_i) f(\omega_{i}) =1+ (\omega_{i}-1)^5 h(\omega_{i}), \ee
for some 
unit $u(\omega_i)$ and $h(x)\in \mathbb Z[x].$
Further, we may take
\be \label{uform} u(x)=\pm x^I(1+x)^J(x^4+1)^{K}\;\;\; I,J,K\geq 0.\ee
\item\label{27normrep3}
If $3\mid h(1)$ in \eqref{genred3} we can write
\be \label{type56}  u(\omega_i) f(\omega_{i}) =1+ (\omega_{i}-1)^7 t(\omega_{i}) \ee
with $t(x)\in\mathbb{Z}[x]$.
\item\label{27normrep4}
If $i=2$ and $3\nmid h(1)$ in \eqref{genred3} then for some $u(x)$ of the form \eqref{uform}, we can take
\be \label{type1form}  (1-\omega_2)u(\omega_2) f(\omega_{2}) =1-\omega_2 + 3\delta (\omega_2) + 3(\omega_{2}-1)^2t(\omega_2) \ee
with $\delta (x)=x$ or $-1$ and $t(x)\in\mathbb{Z}[x]$.
\end{enumerate}
\end{lemma}

\begin{proof}
\ref{27normrep1} Working mod $x^2+x+1$, we may write $f(\omega_1)=a+b(\omega_1-1).$
Note that $3\nmid a$  since $3\nmid N_1(f).$ If $3\nmid b$ we can replace $f(\omega_1)$ by
$$   \omega_1 f(\omega_1)= (a-3b)+(a-2b)(\omega_1-1),\;\;\;\omega_1^2 f(\omega_1)= (3b-2a)+(b-a)(\omega_1-1),   $$
as $a\equiv -b$ or $a\equiv b \bmod 3$ respectively.
That is, we may select $J\in\{0,1,2\}$ so that $\omega_1^J f(\omega_1)=a+3b(\omega_1-1)$.
Now $\pm a\equiv 1$, $2$ or $4 \bmod 9$ and $\pm \omega_1^Jf(\omega_1)=\delta + 3B(\omega_1-1)+9A$ as claimed in \eqref{3normforma}, and \eqref{3normformb} follows easily. Plainly $N_1(f)\equiv \delta^2 \bmod 9$.

\vspace{1ex}
\noindent
\ref{27normrep2} Fix $i\geq2$, and write $x=\omega_i$ and $\pi=1-\omega_i.$
Since $3$ is $\pi^{\phi (3^i)}$ multiplied by a unit and $\phi (3^i)\geq 6,$  it will be enough to work mod $3$.
Since $3\nmid f(1)$ we can write
$$ \pm f(x) = 1 + \sum_{j\geq 1} a_j\pi^j. $$
Since $x=1-\pi$ and $-(1+x) = 1+\pi -3$, we can select $u_1 \in \{1, x, -1-x\}$ so that
$$ \pm u_1f(x)= 1+\sum_{j\geq 2} b_j\pi^j. $$
From $-x(1+x)\equiv 1-\pi^2 \bmod 3$, $x^2(1+x)^2 \equiv 1+\pi^2+\pi^4 \bmod 3$, taking $u_2=1,-x(1+x)$ or $x^2(1+x)^2$ we get
$$  \pm u_1u_2f(x)= 1+\sum_{j\geq 3} c_j\pi^j. $$
With $u_3=1$, $x^3=(1-\pi)^3\equiv 1-\pi^3 \bmod 3$, or $x^6\equiv 1+\pi^3+\pi^6 \bmod 3$, we obtain
$$  \pm u_1u_2u_3f(x)= 1+\sum_{j\geq 4} d_j\pi^j. $$
Finally, with $u_4=1$, $-x^6(1+x)^2(1+x^4)\equiv 1+\pi^4 \bmod \langle 3,\pi^5\rangle$,  or $x^{12}(1+x)^4(1+x^4)^2\equiv 1-\pi^4 \bmod \langle 3,\pi^5\rangle$, we get
$$  \pm u_1u_2u_3u_4 f(x)= 1+\sum_{j\geq 5} e_j\pi^j, $$
giving \eqref{genred3} and \eqref{uform}.

\vspace{1ex}
\noindent
\ref{27normrep3} If $3\mid e_5$ then this becomes  $  \pm u_1u_2u_3u_4 f(x)= 1+\sum_{j\geq 6} t_j\pi^j. $
With $u_5=1$ or $-x^3(1+x)^3\equiv 1-\pi^6 \bmod 3\pi^2$ or $x^6(1+x)^6\equiv 1+\pi^6 \bmod \langle 3\pi^2,\pi^{12}\rangle$ we get
$$  \pm u_1u_2u_3u_4u_5 f(x)= 1+\sum_{j\geq 7} s_j\pi^j, $$
and the form \eqref{type56}.

\vspace{1ex}
\noindent
\ref{27normrep4} If $i=2$ and $3\nmid e_5$ then using that $\pi^6$ is 3 times a unit we get 
$  \pm u_1u_2u_3u_4 (1-x)f(x)= \pi  +3s(x)$  with $3\nmid s(1)$, and we can write this as $\pi+3\delta(x)+A\pi^7+3\pi^2t_1(x)$ with $\delta(x)=x$ or $-1$ as $s(1)$ is $1$ or $-1 \bmod 3$, $A=0$ or $\pm 1,$ and $t_1(x)\in\mathbb{Z}[x]$. 
Taking $u_5$ as above we get  $\pm u_1u_2u_3u_4 u_5\pi f(x)=\pi+3\delta(x) +3\pi^2 t_2(x)$
and the form \eqref{type1form}.
\end{proof}

We shall say that \eqref{3normforma} and \eqref{genred3} are \textit{Type~1} representations of $f(\omega_i)$ if $3\nmid A$ or $3\nmid h(1)$ respectively, and \textit{Type~2} representations if $3\mid A$ or  $3\mid h(1)$. 
For $i\geq2$, we say a Type~2 representation \eqref{type56} is \textit{Type~3} or \textit{Type~4} depending on whether $3\nmid t(1)$ or $3\mid t(1)$ respectively.

\begin{lemma}\label{notboth}
Suppose that $f(x)\in \mathbb Z[x]$ with $3\nmid N_1(f)$, and let $i\geq1$.
\begin{enumerate}[label=(\roman*)]
\item\label{notboth1}
$f(\omega_i)$ cannot have both a Type~1 and a Type~2 representation.
\item\label{notboth2}
If Type~2, then $f(\omega_i)$  cannot have both a Type~3 and a Type~4 representation.
\item\label{notboth3}
The conjugates $f(\omega_i^j)$, $3\nmid j$, will all be of Type~1 or all of Type~2, and if Type~2, then all will be of Type~3 or all will be of Type~4.
\item\label{notboth4}
A unit has only Type~2 and Type~4 representations.
\end{enumerate}
\end{lemma}

\begin{proof}
\ref{notboth1} For $i=1$, suppose that $\pm \omega_1^{J_1}f(\omega_1)= \delta + 3A_1(\omega_1-1)+9B_1$, $3\nmid A_1$ and
$\pm \omega_1^{J_2}f(\omega_1)= \delta + 3A_2(\omega_1-1)+9B_2$, $3\mid A_2.$ So $(1-\omega_1)^3\parallel \pm \omega_1^{J_1} \mp \omega_1^{J_2},$ but this is either $0$ or divisible by exactly one $(1-\omega_1)$ or not divisible by $(1-\omega_1).$
Suppose then that $i\geq 2$ and that we had $u_j(\omega_i)f(\omega_i)=1+(\omega_i-1)^5 h_j(\omega_i),$ $j=1,2,$ with $3\nmid h_1(1)$ and $3\mid h_2(1)$.
Then 
$$u_1(\omega_i)-u_2(\omega_i) = (\omega_i-1)^5 \left( u_2(\omega_i)h_1(\omega_i)- u_1(\omega_i)h_2(\omega_i)\right)$$ with $(\omega_i-1)\nmid h_1(\omega_i)$, $(\omega_i-1)\mid h_2(\omega_i).$
But then $(\omega_i-1)^5\parallel (u_1(\omega_i)-u_2(\omega_i))$ contradicting Lemma~\ref{Newman}.

\vspace{1ex}
\noindent
\ref{notboth2} In the same way having both a Type~3 and Type~4 is ruled out since we cannot have  $(\omega_i-1)^7\parallel (u_1(\omega_i)-u_2(\omega_i)).$

\vspace{1ex}
\noindent
\ref{notboth3} For $i\geq 2$ observe that the reduction \eqref{genred3} yields $u(\omega_i^j)f(\omega_i^j) =1 + (\omega_i-1)^5h_j(\omega_i)$ with $h_j(x)=(1+x+\cdots + x^{j-1})^5h(x)$, and $3\mid h_j(1)$ exactly when $3\mid h(1)$.
The argument is similar for Types~3 or~4. For $i=1$ the conjugate of \eqref{3normforma} has
\be \label{i=1}  \delta +3A(\omega_1^2-1)+9b= \delta - 3A(\omega_1-1)+9(b-A). \ee
This establishes \ref{notboth3}.

\vspace{1ex}
\noindent
\ref{notboth4} Note that $1$ is written in Type~2 and Type~4 form, ruling out a unit being Type~1 or Type~3.
\end{proof}

We shall say that an integer $m$ is a \textit{$3^i$-norm} if $m=N_i(f)$ for some $f\in\mathbb Z[x]$.
Note this says that $m$ is a product of $q^f\equiv 1 \bmod 3^i.$
From Lemma~\ref{27normrep} we immediately obtain the following.

\begin{lemma}\label{lem3norms}
If $m$ is a $3$-norm, then for some $A,B\in \mathbb Z$ we have
\be \label{1norms} 3m=N_1(\delta (1-x) +  9A + 9B(x-1) ) \ee
with $\delta=1,2$ or $4$ as $m\equiv 1$, $4$ or $7 \bmod 9$.
If $m$ is a $3^i$-norm with $i\geq 2,$  then for some $h(x)$ in $\mathbb Z [x]$ we have
\be \label{inorms}  m=N_i(1+(x-1)^5h(x)),\ee
and if $3\mid h(1)$ there is a $t(x)$ in $\mathbb Z[x]$ with
\be \label{inorms56}  m=N_i(1+(x-1)^7t(x)).\ee
If  $m=N_2(1+(x-1)^5h(x))$ with  $3\nmid h(1),$ then  we can write
\be\label{2norms} 3m = N_2\left(1-x+3\delta (x )+ 3(x-1)^2 t(x)\right), \;\;\; \delta(x)=x \text{ or } -1.\ee
\end{lemma}

We say that $m$ has a $3$-norm representation of Type~1 or Type~2 depending on whether $3\nmid A$ or $3\mid A$ respectively in \eqref{1norms}.
Similarly, for $i\geq2$ we say that $m$ has a $3^i$-norm representation of Type~1 or Type~2 as $3\nmid h(1)$ or $3\mid h(1)$ in \eqref{inorms}; in the latter case a Type~3 or Type~4 as $3\nmid t(1)$ or $3\mid t(1)$ in \eqref{inorms56}.

If $m$ has a $3^i$-norm representation of Type~1 or~3, then the same is true for the lower norms.

\begin{lemma}\label{lemlowernorm}
Suppose $m\in\mathbb{Z}$ and $i\geq2$.
\begin{enumerate}[label=(\roman*)]
\item\label{lemlowernorm1}
If $m$ has a $3^i$-norm representation of Type~1 or Type~2, then $m$ has the same type of $3^j$-norm representation 
for $2\leq j \leq i$, and ultimately
$$ 3m=  N_1(1-x + 9(A+B(1-x))), $$
with $3\nmid A$ for Type~1 and $3\mid A$ for Type~2.
\item\label{lemlowernorm2}
If $m$ has a $3^i$-norm representation of Type~3 or Type~4, then $m$ has the same type of $3^j$-norm representation for $2\leq j \leq  i.$
\end{enumerate}
\end{lemma}

\begin{proof}
\ref{lemlowernorm1} Suppose that $i\geq 2$. Applying Lemma \ref{reduction} we have $ N_i(F)=N_{i-1}(g)$ 
with $g(x)\equiv F(x) \bmod 3$.
When $i-1\geq 2$ we have $(1-\omega_{i-1})^6\mid 3$ and
$$ N_i(1+(x-1)^5h(x))=N_{i-1}(1 + (x-1)^5(h(x) + (x-1)t(x))), $$
with this a $3^{i-1}$-norm Type~1 or~2 as the original was a $3^i$-norm Type~1 or~2.
When $i=2,$ from~\eqref{2norms} we write 
$$3m=N_2(1-x-3h(x)),\;\; h(x)=h_0(x^3)+xh_1(x^3)+x^2h_2(x^3),$$
and use the explicit formula \eqref{defggg} with $f_0(x)=1-3h_0(x)$, $f_1(x)=-1-3h_1(x)$, $f_2(x)=-3h_2(x)$ and $3m=N_1(g(x))$ with
\begin{align*} g(x) & \equiv 1-9h_0(x) - x(1 + 9h_1(x) + 9h_2(x))  \bmod 27 \\
   &   \equiv 1-x -9h(1) \bmod 9(x-1). 
\end{align*}
Notice, reducing mod $x^2+x+1,$  that we can write $9(x-1)t(x)=9(x-1)(\alpha +\beta (x-1))=-27\beta +9(x-1)(\alpha -3\beta)$ and reduce to the form $g(x)=1-x -9A+9B(x-1)$ with $3\mid A$ exactly when $3\mid h(1)$.

\vspace{1ex}
\noindent
\ref{lemlowernorm2} Suppose that $i\geq 3$ and $m=N_i(F)$ with $F(x)=1+(x-1)^7h_1(x).$
Since $(\omega_i-1)^3=(\omega^3_i-1)u(\omega_i)$ with $u(\omega_i)$ a unit, we can replace this by $F(x)=1+(x^3-1)^2(x-1)h(x),$  with $3\nmid h(1)$ for Type~3 and $3\mid h(1)$ for Type~4.
Writing $h(x)=h_0(x^3)+xh_1(x^3)+x^2h_2(x^3)$ we get $F(x)=f_0(x^3)+xf_1(x^3)+x^2f_2(x^3)$ with
\begin{align*}  f_0(x) & =1+(x-1)^2(xh_2(x)-h_0(x)),\\
      f_1(x)&=(x-1)^2(h_0(x)-h_1(x)),\\
f_2(x)&=(x-1)^2(h_1(x)-h_2(x)), \end{align*}
and \eqref{defggg} becomes
\begin{align*}
g(x)  & = 1+(x-1)^6(x^3h_2(x)^3-h_0(x)^3)  + x(x-1)^6(h_0(x)^3-h_1(x)^3) \\
& +x^2(x-1)^6(h_1(x)^3-h_2(x)^3)+3(x-1)^2g_1(x)\\ & =  1+(x-1)^7(h_0(x)^3+xh_1(x)^3+x^2h_2(x)^3) + 3(x-1)^2g_1(x) \\
& =1+(x-1)^7 h(x)+3(x-1)^2g_1(x),
\end{align*}
which will be Type~3 if $3\nmid h(1)$ and Type~4 if $3\mid h(1).$
\end{proof}

We first show that we can achieve $3^4m$ for $m$ any $9$-norm of Type~1.

\begin{lemma}\label{lem9norm}
Let $m\in\mathbb{Z}$.
\begin{enumerate}[label=(\roman*)]
\item\label{lem9norm1}
If $m$ is a $9$-norm with a Type~1 representation then there is an $F(x) \in \mathbb Z[x]$ with 
$F(1)=N_1(F)=N_3(F)=3,$ $N_2(F)=3m$ and $M_{27}(F)=3^4m.$
\item\label{lem9norm2}
If $m$ is a $27$-norm with a Type~1 representation then there is also an $F(x) \in \mathbb Z[x]$ with 
$F(1)=N_1(F)=N_2(F)=3,$ $N_3(F)=3m$ and $M_{27}(F)=3^4m.$
\end{enumerate}
\end{lemma}

\begin{proof} 
\ref{lem9norm1} Using Lemma~\ref{lem3norms}, we write $3m=N_2(1-x+ 3\delta(x) +3(x-1)^2t(x))$ with $\delta(x)=-1$ or $x,$ and observe that we can write 
$$1-\omega_1 + 3\delta (\omega_1) +3(\omega_1-1)^2t(\omega_1) =1-\omega_1 + 3\delta (\omega_1) -9\omega_1 t(\omega_1).$$
Select integers $\ell\geq 1$ and $\lambda$ so that $\ell ( \delta (1)+3t(1))=1+9\lambda$, so $3\nmid\ell$.
Set
\be \label{lambdered} F(x) =\left(\frac{ x^{\ell}-1}{x-1}\right) f(x) - \lambda \left(\frac{ x^{27}-1}{x-1}\right) , \ee
with
$$ f(x)=1-x + \Phi_{27}(x)\left(\delta (x)+(x-1)^2t(x)\right)+xt(x)\Phi_9(x)\Phi_{27}(x).$$
Then $N_3(F)=N_3(1-x)=3$, $N_2(F)=3m,$ $N_1(F)=N_1(1-x+3\delta(x))=3$ and $F(1)= \ell( 3\delta(1)+9t(1))-27\lambda=3.$

\vspace{1ex}
\noindent
\ref{lem9norm2} Suppose now that $m=N_3(1+(x-1)^5h(x))$ with $3\nmid h(1)$.
Using
\begin{gather*}
(\omega_2-1)^6 = -3\omega_2(\omega_2^2-\omega_2+1)(2\omega_2^2-3\omega_2+2),\\
3 = (\omega_2-1)^6(10+11\omega_2+7\omega_2^2+10\omega_2^3+4\omega_2^4-4\omega_2^5), 
\end{gather*}
and setting $\delta(x)=1$ if $h(1)\equiv 1 \bmod 3$ and $-x$ if $h(1)\equiv -1 \bmod 3$,
we can write
\begin{align*}
 (\omega_2-1)\left(1+(\omega_2-1)^5h(\omega_2)\right) =(\omega_2-1)-3h(1)+\sum_{j\geq 1} 3a_j(\omega_2-1)^j \\
   = (\omega_2-1)-3\delta(\omega_2) + b_1(\omega_2-1)^7 +\sum_{j\geq 2} 3b_j(\omega_2-1)^j,
\end{align*}
with $b_1=0,-1$ or $1.$ Taking $u(x)=1$, $-x^3(1+x)^3$ or $x^6(1+x)^6$ as $b_1=0$, $-1$ or $1$, we have $u(\omega_2)\equiv1+b_1(\omega_2-1)^6 \bmod 3(\omega_2-1)^2$ and
$$ (\omega_2-1)\left(1+(\omega_2-1)^5h(\omega_2)\right) =(\omega_2-1)u(\omega_2)  -3\delta(\omega_2)-3(\omega_2-1)^2g(\omega_2) $$
for some $g(x)\in\mathbb Z[x].$
For some $t(x)\in\mathbb Z[x]$ we can write
$$ (\omega_1-1)^6h(\omega_1)+3(\omega_1-1)^2g(\omega_1)= 9t(\omega_1). $$
For integers $\ell>0$ and $\lambda$ with $\ell (\delta(1)-3t(1))=1+9\lambda$ take $F(x)$ as in \eqref{lambdered}
with
$$  f(x)= (x-1)(1+(x-1)^5h(x)) + \delta(x)\Phi_{27}(x)+(x-1)^2\Phi_{27}(x)g(x) - \Phi_{27}(x)\Phi_9(x) t(x). $$
Then $F(1)=\ell (3\delta(1)-9t(1)) -27\lambda  =3,$ $N_1(F)=N_1(f)=N_1(x-1+3\delta(x))=3,$ $N_2(F)=N_2((x-1)u(x))=3$
and $N_3(F)=3m.$
\end{proof}

We next show that we can achieve $3^5m$ when $m$ is $3$-norm of Type~1, or $9$-norm of Type~3.

\begin{lemma}\label{lem243s}
Let $m\in\mathbb{Z}$.
\begin{enumerate}[label=(\roman*)]
\item\label{lem243s1}
If $m$ is a $3$-norm with a Type~1 representation, then there is an $F(x)\in\mathbb Z[x]$ with $F(1)=9,$ $N_1(F)=3m,$ $N_2(F)=N_3(F)=3$ and $M_{27}(F)=3^5m.$
\item\label{lem243s2}
If $m$ is a $9$-norm with a Type~3 representation, then there is an $F(x)\in\mathbb Z[x]$ with $F(1)=9,$ $N_2(F)=3m,$ $N_1(F)=N_3(F)=3$ and $M_{27}(F)=3^5m.$
\item\label{lem243s3}
If $m$ is a $27$-norm with a Type~3 representation, then there is also an $F(x)\in\mathbb Z[x]$ with $F(1)=9,$ $N_3(F)=3m,$ $N_1(F)=N_2(F)=3$ and $M_{27}(F)=3^5m.$
\end{enumerate}
\end{lemma}

\begin{proof}
\ref{lem243s1} Suppose that $m$ is a $3$-norm of Type~1.
Then $3m=N_1(\delta(1-x)+9A+9B(x-1))$ with $3\nmid A$ by Lemma~\ref{lem3norms}.
Take a positive integer $\ell$ and integer $\lambda$ such that $A\ell =1+3\lambda, $ and set $\delta(x)=1$, $x^3+1$ or $(x^3+1)^2$
as $\delta=1$, $2$ or $4$. Let $F(x)$ be as in \eqref{lambdered} with
$$ f(x)= \delta(x)(1-x)+\Phi_9(x)\Phi_{27}(x) (A+B(x-1)). $$
This has $N_3(F)=N_2(F)=3$, $N_1(F)=3m$ and $F(1)=9A\ell-27\lambda=9.$

\vspace{1ex}
\noindent
\ref{lem243s2} Suppose that $m$ is a $9$-norm of Type~3.
Then $3m=N_2(1-x+3(x-1)^2 t(x))$, $3\nmid t(1)$.
Taking $\ell\geq 1$ and $\lambda$ such that $\ell t(1)=1+3\lambda$ and \eqref{lambdered} with 
$$ f(x)=1-x + \Phi_{27}(x)(x-1)^2 t(x)  +x\Phi_9(x)\Phi_{27}(x) t(x), $$
we have $N_3(F)=3$, $N_2(F)=3m$, $N_1(F)=3$, $F(1)=9.$

\vspace{1ex}
\noindent
\ref{lem243s3} Suppose that $m$ is a $27$-norm of Type~3 so that we can write $3m=N_3(1-x+(1-x)^8h(x))$, $3\nmid h(1)$.
With $(\omega_2-1)^6=3t(\omega_2)$ we take $\ell\geq 1$ with $-\ell h(1)t(1)=1+3\lambda$ and \eqref{lambdered} with
\begin{align*}
f(x) &=(1-x)+(1-x)^8h(x) -\Phi_{27}(x)t(x)h(x)(1-x)^2\\
&\qquad -x\Phi_9(x)\Phi_{27}(x)h(x)(t(x)+x(1-x)^4). 
\end{align*}
Then $N_3(F)=3m,$  $N_2(F)=3$, $N_1(F)=3,$ $F(1)=9.$
\end{proof}

\begin{lemma}\label{lempf3normrep}
Suppose $q\neq3$ is prime, $i\geq1$, and $f$ is the minimal positive integer with $q^f\equiv1\bmod 3^i$.
Then 
\begin{enumerate}[label=(\roman*)]
\item\label{lempf3normrep1}
$q^f$ cannot have $3^i$-norm representations of both Type~1 and Type~2. 
\item\label{lempf3normrep2}
A product of $q^f$ of Type~2 must be Type~2, a product of one Type~1 and the rest Type~2 will be Type~1.
A product of two or more Type~1 will have both Type~1 and Type~2 representations.
Analogous statements hold for Types~3 and~4.

\end{enumerate}
\end{lemma}

\begin{proof} 
\ref{lempf3normrep1} Suppose that  $q^f$ is the $3^i$-norm of a $g_1(\omega_i)$ with a Type~1 representation and $g_2(\omega_i)$ with a Type~2 representation.
Then up to units these must be conjugates and  $g_2(\omega_i^j)=u g_1(\omega_i)$ gives us a representation of both types for $g_1(\omega_i)$, contradicting Lemma~\ref{notboth}.

\vspace{1ex}
\noindent
\ref{lempf3normrep2} For $i=1$ the product of two $(\delta_j+ 3A_j(\omega_1-1)+9b_j),$ $j=1,2,$  will take the form
$\pm \delta_3+ 3(A_1\delta_2+A_2\delta_1)(\omega_1-1)+9(C\omega_1+D)$ where $\delta_1\delta_2\equiv \pm \delta_3 \bmod 9$, and if $3\mid A_1$ we will have $3\mid (A_1\delta_2+A_2\delta_1)$ if and only if $3\mid A_2$. If $3\nmid A_1A_2$ then the choice of sign determines whether $3\mid  (\pm A_1\delta_2+A_2\delta_1)$ and from \eqref{i=1} the choice of conjugate  gives both types. For $i\geq 2,$ expanding $h(x)=h(1)+\sum_{j\geq 1} a_j(x-1)^j$ it is clear that the 
product of $1+(x-1)^5h_1(x)$ and $1+(x-1)^5h_2(x)$
can be written $1+(x-1)^5h_3(x)$ with $h_3(1)=h_1(1)+h_2(1).$
Plainly $3\mid h_3(1)$ if $3$ divides both $h_1(1)$ and $h_2(1)$ and $3\nmid h_3(1)$
if $3$ divides only one of them.
Hence a product of $q^f$ of Type~2 will be Type~2.
Moreover  by Lemma~\ref{notboth}\ref{notboth3} this is independent of the conjugates used and  Lemma~\ref{notboth}\ref{notboth1} 
rules out this also being Type~1.
Similarly a product of one $q^f$ of Type~1 and the rest Type~2 will be Type~1 and not Type~2.
If $3$ divides neither then taking conjugates $x\mapsto x^k$
 in the first replaces $h_1(1)$ by $k^5 h_1(1)\equiv k h_1(1) \bmod 3$ and the choice of $k$ can make either $3\mid h_3(1)$ or $3\nmid h_3(1)$.
The argument is similar for Type~3 and Type~4.
\end{proof}

It remains to characterize the $3$-norm, $9$-norm  and $27$-norm types for the $q^f.$ 

\begin{lemma}\label{lem9normts}
Let $q$ be a prime.
\begin{enumerate}[label=(\roman*)]
\item\label{lem9normts1}
If $q\equiv -1 \bmod 3$ then $q^2$ is $3$-norm Type~2.
\item\label{lem9normts2}
If $q\equiv 1 \bmod 9$ then $q$ is $9$-norm  Type~1 if $q\in \mathcal{T}^{(3)}_1,$ Type~2 if $q\in \mathcal{T}^{(3)}_2$ and Type~3 or~4 as $q\in  \mathcal{T}^{(3)}_3$ or $\mathcal{T}^{(3)}_4$.
\item\label{lem9normts3}
If  $q\equiv 8 \bmod 9$ then $q^2$ is $9$-norm Type~2 and Type~4.
\item\label{lem9normts4}
If $q\equiv 4$ or $7 \bmod 9$ then $q^3$ is $9$-norm Type~2 and Type~4.
\item\label{lem9normts5}
If $q\equiv 2$ or $5 \bmod 9$ then $q^6$ is $9$-norm Type~2 and Type~4.
\item\label{lem9normts6}
A minimal $q^f$  is $27$-norm Type~1 or Type~3 if $q\equiv 1 \bmod 27,$ $f=1,$ with  $q\in \mathcal{T}^{(3)}_1$ or $ \mathcal{T}^{(3)}_3$ respectively. Otherwise $q^f$ is  $27$-norm Type~2 and 4.
\item\label{lem9normts7}
$1$ is $3$-norm Type~2, and for $i\geq 2$ is $3^i$-norm Type~2 and Type~4 only.
\end{enumerate}
\end{lemma}

\begin{proof}
\ref{lem9normts1} For $q\equiv 2 \bmod 3$ we have $q^2=N_1(\pm q)$ with $\pm q$ of the form $1+9k$, $2+9k$ or $4+9k,$ yielding a $3$-norm Type~2 representation in \eqref{1norms}.

\vspace{1ex}
\noindent
\ref{lem9normts2} and \ref{lem9normts3}
By Lemma~\ref{lemlowernorm}\ref{lemlowernorm1} determining whether $q^f$ is $9$-norm Type~1 or Type~2 reduces to determining the $3$-norm type.
When $q\equiv 1 \bmod 9$, a prime can be $9$-norm Type~1 or Type~2 and those of Type~2 can be Type~3 or Type~4.
For $q\equiv -1 \bmod 9$ we know that  $q^2$ is $3$-norm Type~2.
We know that such a prime splits into three factors of norm $q$ in $\mathbb  Z[\omega_2+\omega_2^{-1}]$ (e.g., \cite[ex.\ 4.12]{Marcus}), and this remains the factorization in $\mathbb Z[\omega_2].$ Hence we must have $q^2=N_2(f)$ with $f(\omega_2)$ real and
$$ u(\omega_2)f(\omega_2) = 1+(\omega_2-1)^7h(\omega_2),\;\;\; \overline{u(\omega_2)}f(\omega_2) = 1-(\omega_2-1)^7\omega_2^2\:\overline{h(\omega_2)}. $$
If $3\nmid h(1)$ then $(\omega_2-1)^7 \parallel (\overline{u(\omega_2)}-u(\omega_2))$, but this is ruled out by Lemma~\ref{Newman}. Hence $3\mid h(1)$ and $q^2$ is Type~4.

\vspace{1ex}
\noindent
\ref{lem9normts4}
We have $q=N_1(1\pm 3 +3(x-1)h(x))$ and $q^3=N_2(1\pm 3 +3(x^3-1)h(x^3)).$
Since $3=-(\omega_2-1)^6 \bmod (\omega_2-1)^9$ we can write $q^3=N_2(1\mp (x-1)^6 +(x-1)^9h_1(x)),$ plainly Type~2.
Moreover, multiplying by the unit $u_1=-\omega_2^3(\omega_2+1)^3=1-(\omega_2-1)^6 \bmod (1-\omega_2)^8$ or $u_1^2=1+(\omega_2-1)^6 \bmod (1-\omega_2)^8$ we can replace this by $1+(x-1)^8h_2(x)$ which is plainly Type~4.

\vspace{1ex}
\noindent
\ref{lem9normts5}
We have $q^6=N_2(-q)=N_2(1+3k)=N_2(1+(x-1)^6h(x)),$ a $9$-norm Type~2.
Moreover $-q=1-3+9k=1+(\omega_2-1)^6 \bmod (\omega_2-1)^9$ or $1-6+9k=1-(\omega_2-1)^6 \bmod (\omega_2-1)^9$.
Selecting $j=1$ or $2$ as appropriate we find $-u_1^j q=1 \bmod (\omega_2-1)^8$, so $q^6$ is $9$-norm Type~4.

\vspace{1ex}
\noindent
\ref{lem9normts6}
Suppose that $q^f$ is $27$-norm Type~1 or Type~3.
Then by Lemma~\ref{lemlowernorm} $q^f$ is $9$-norm Type~1 or~3 and  by Lemma~\ref{lempf3normrep}\ref{lempf3normrep2} the minimal power $q^{f'}\equiv 1 \bmod 9$ 
must  also be Type~1 or~3.
So $f'=1$, $q\equiv 1 \bmod 9$ with $q$ in $ \mathcal{T}^{(3)}_1$ or $ \mathcal{T}^{(3)}_3.$
If $q\equiv 1 \bmod 27$ then $f=1$ and we are done (the 27-norm type is the same as the 9-norm type).
Otherwise $q\equiv 10$ or $19 \bmod 27$ and $f=3.$
Now we can write
$q=N_{2}(1+(1-x)^5h(x)),$ and $q^3=N_3(1+(1-x^3)^5h(x^3)).$ Since $(1-x^3)^5\equiv (1-x)^{15} \bmod 3$,  and $3$ can be replaced by $(1-x)^{18}h_2(x),$ this gives $q^3$ a $27$-norm Type~4 representation, and a contradiction.
  
\vspace{1ex}
\noindent
\ref{lem9normts7} This is evident from \eqref{inorms} and Lemma~\ref{notboth}\ref{notboth4}.
\end{proof}

\begin{proof}[Proof of Theorem~\ref{thmZ27}]
\ref{thmZ27a} By Lemmas~\ref{lem9norm} and~\ref{lem9normts}, if $q\equiv 1 \bmod 9$ is Type~1, then we can achieve $3^4 q$ as a $\mathbb{Z}_{27}$-measure, and by closure under multiplication we can achieve $3^4 q m$ for any integer $m$ with $3\nmid m$.
It remains to show that if $3^4 m$ with $3\nmid m$ is achieved as a $\mathbb{Z}_{27}$-measure then it must be divisible by a Type~1 prime $q\equiv 1 \bmod 9$. More precisely we rule out $N_2(F)=3m_2,$ $N_3(F)=3m_3$ with $m_2$ only $9$-norm Type~2 and $m_3$ only 27-norm Type~2; in particular by  Lemmas~\ref{lempf3normrep} and~\ref{lem9normts} our $M_{27}(F)=3^4m$, $3\nmid m,$ 
must contain a Type~1 prime $q\equiv 1 \bmod 9$ dividing $N_2(F)$ or a Type~1 prime $q\equiv 1 \bmod 27$ dividing $N_3(F).$

Suppose that $m_3$ is only $27$-norm Type~2, then from Lemma~\ref{Lemma1} and the form \eqref{type56} for $F(\omega_3)/(1-\omega_3)$, and using $F(1)=3m_0,$ $3\nmid m_0,$ we may take
$$ F(x) = u_1(x)(1-x -(x-1)^8t(x) ) + m_0\Phi_{27}(x) + (1-x)\Phi_{27}(x)s(x) $$
with $s(x)\in\mathbb{Z}[x]$, where $u_1(x)$ is a product of units $u(x^j)$, $3\nmid j$, each $u(x)$ of the form \eqref{uform}.
But since $3\nmid m_0$ this makes
\begin{align*}  F(\omega_2) & = u_1(\omega_2) (1-\omega_2 -(\omega_2-1)^8t(\omega_2) )+3m_0+3(1-\omega_2)s(\omega_2) \\
 & = u_1(\omega_2) (1-\omega_2) \left(1 + (\omega_2-1)^5 h(\omega_2)\right), \;\; 3\nmid h(1),
\end{align*}
and $N_2(F)=3m_2$, where $m_2$ has a  $9$-norm representation of Type~1. 

\vspace{1ex}
\noindent
\ref{thmZ27b} Similarly, by Lemmas~\ref{lem243s} and~\ref{lem9normts} we can achieve $3^5mq,$ $3\nmid m,$ for any $q\equiv 1 \bmod 3$ of Type~1, or $q\equiv 1 \bmod 9$ of  Type~3.
By  Lemmas~\ref{lempf3normrep} and~\ref{lem9normts} any remaining $3^5m,$ $3\nmid m,$ must have $N_i(F)=3m_i$ 
with $m_1$  only $3$-norm Type~2,  $m_2$  only $9$-norm Type~4 and $m_3$ only $27$-norm Type~4.
Hence with $F(1)=3m_0$, $3\nmid m_0,$ we can take 
$$ F(x)=u_1(x)(1-x- (x-1)^9h(x)) + 3m_0 \Phi_{27}(x) +(x-1)\Phi_{27}(x) t(x) $$
where $u_1(x)$ is a product of units $u(x^j)$ of the form \eqref{uform} with $3\nmid j$.
Since these are units  for $\omega_2$ we can replace $t(x)$ by $t_1(x)u_1(x) \bmod \Phi_9(x)$ and take
\begin{align*}
 F(x) &= u_1(x)\left(1-x- (x-1)^9h(x) + (x-1)\Phi_{27}(x)t_1(x)\right)  + 3m_0 \Phi_{27}(x) \\
 &\qquad +(x-1)\Phi_9(x)\Phi_{27}(x)t_2(x).
\end{align*}
Expanding $t_1(x)=a_0+a_1(x-1)+ (x-1)^2t_4(x)$ we get 
$$ N_2(F)= N_2\left( 1-x+3a_0(x-1)+3a_1(x-1)^2+ 3(x-1)^3h_1(x)\right) $$
and
$$ N_1(F)=N_1\left( 1-x+3a_0(x-1) -9a_1   +9m_0u_1(1) +9(x-1)h_2(x) \right).$$
We can reduce $a_0=0,\pm 1 \bmod 3$ and get
$$ N_1(F)=N_1\left(\pm \delta ( 1-x)-9a_1   +9m_0u_1(1) +9(x-1)h_3(x) \right),$$
and we can deduce that $3\nmid a_1$, since $m_1$ is not Type~1.
If $3\mid a_0$ then $3\nmid a_1$ contradicts that $m_2$ is not Type~3.
If $3\nmid a_0$, using $3 \equiv -(\omega_2-1)^6 \bmod (\omega_2-1)^9$, and multiplying by  $-x^3(1+x)^3 \equiv 1-(x-1)^6 \bmod 3(x-1)^2$, or $x^6(1+x)^6 \equiv 1+(x-1)^6 \bmod 3(x-1)^2$, we obtain a Type~3:
\begin{align*}  N_2(F) &= N_2\left( 1-x\pm (x-1)^7+3a_1(x-1)^2+ 3(x-1)^3h_4(x)\right) \\
  & =N_2 \left( 1-x+3a_1(x-1)^2+ 3(x-1)^3h_5(x)\right),
\end{align*}
which is again a contradiction. 
\end{proof}

In the proof of Theorem~\ref{thmZ25}, we showed that any $M_{25}(F)=5^3m$, $5\nmid m,$ must have $m$ divisible by a (Type~1) prime $q\equiv 1 \bmod 5$ dividing $N_1(F)$, or a (Type~1) prime $q\equiv 1 \bmod 25$ dividing $N_2(F).$
Similarly in the proof of Theorem~\ref{thmZ27} we actually proved that any $M_{27}(F)=3^4m$, $3\nmid m,$ must have $m$ divisible by a (Type~1) prime $q\equiv 1 \bmod 9$ dividing
$N_2(F)$ or a (Type~1) prime $q\equiv 1 \bmod 27$ dividing $N_3(F)$, and any $M_{27}(F)=3^5m$, $3\nmid m,$ must have $m$ divisible by a (Type~1) prime $q\equiv 1 \bmod 3$ dividing
$N_1(F)$, or a (Type~1 or~3) prime $q\equiv 1 \bmod 9$ dividing $N_2(F)$, or a (Type~1 or~3) prime $q\equiv 1 \bmod 27$ dividing $N_3(F)$.
It is tempting to ask whether one actually needs powers $q_i^{a_i}$ of primes $q_i\not\equiv 1 \bmod p^i$ in Theorem~\ref{thmpowers}.

\begin{question}
In Theorem~\ref{thmpowers}\ref{Zpta} can we guarantee a prime factor $q_i\equiv 1 \bmod 3^i$
in $2t-j$ of the $N_i(F)$,  $1\leq i\leq t$, and similarly for $t-2$ of  the $N_i(F)$, $2\leq i\leq t$, in \ref{Zptb}?
\end{question}

In the next section we explore this question computationally.

\section{Computations}\label{secComputations}

This research proceeded in concert with a number of experiments and computations on integral circulant determinants.
For example, we computed the $\mathbb{Z}_{27}$-measure of all polynomials $F(x)$ with $\{-1,0,1\}$ coefficients, degree at most $26$ and $F(1)=3$ or $9$.
This found over a million different values of $m$ for which there was a polynomial $F(x)$ with $F(1)=\abs{N_1(F)}=\abs{N_2(F)}=3$ and $\abs{N_3(F)}=3m$ with $3\nmid m$, so $\abs{M_{\mathbb{Z}_{27}}(F)}=3^4m$, and more than $5.5$ million where $F(1)=9$, $\abs{N_1(F)}=\abs{N_2(F)}=3$, and $\abs{N_3(F)}=3m$, $3\nmid m$, so $\abs{M_{\mathbb{Z}_{27}}(F)}=3^5m$.
In each case the integer $m$ always possessed at least one prime divisor $q \equiv 1 \bmod 27$.
Such observations led to the development of some of the results in this paper.

Here we report briefly on some additional computations related to the topics of this research.
For a number of groups $\mathbb{Z}_{p^t}$, we computed $M_{\mathbb{Z}_{p^t}}(F)$ for certain families of polynomials $F(x)\in\mathbb{Z}[x]$ having $p\mid F(1)$ in order to investigate the multiples of $p$ attained in $\mathcal{S}(\mathbb{Z}_{p^t})$.
We describe some computations connected to each of our three main results here.

\subsection{Computations connected to Theorem~\ref{thmpowers}}\label{subsecCalcs1}

Some computations related to our first main result indicated that odd values for $a_i$ larger than $1$ did not appear to be necessary for several families.
In $\mathbb{Z}_{81}$, by Lemma~\ref{Lemma1}, if $M_{\mathbb{Z}_{81}}(F) = 3^7m$ with $3\nmid m$, then $27 \parallel F(1)$ and $3 \parallel N_i(F)$ for $1\leq i\leq4$.
A search over $\mathbb{Z}_{81}$ to study such $m$, similar to the ones performed here for $\mathbb{Z}_{25}$ and $\mathbb{Z}_{27}$, would be well out of reach: there are more than $2\cdot10^{20}$ polynomials $F(x)$ with $\deg(F)\leq80$, $F(1)=27$ and nonzero constant and linear coefficients, when using just $\{0,1\}$ coefficients.
However, a more directed search allowed us to determine a substantial number polynomials with $M_{\mathbb{Z}_{81}}(F)=3^7 m$ and $3\nmid m$.
For this we first determined some restrictions on the coefficients to guarantee that $N_1(F) = N_2(F) = 3$ as well as $F(1)=27$.
For the $N_1$ restriction, one may show by elementary means that a quadratic polynomial $f(x)=c_0+c_1x+c_2x^2\in\mathbb{Z}[x]$ with $f(1)=27$ and $N_1(f)=3$ must have $\{c_0,c_1,c_2\} = \{8,9,10\}$.
We set $c_0=10$, $c_1=9$ and $c_2=8$.
Next, let $g(x)=b_0+b_1 x+\cdots+b_8 x^8$: we require that $N_2(g)=3$ and simultaneously $b_0+b_3+b_6=10$, $b_1+b_4+b_7=9$,  $b_2+b_5+b_8=8$.
These requirements produce a polynomial equation in six variables $h(b_0,\ldots,b_5)=0$ for which any solution in integers gives rise to a polynomial $g(x)$ having $g(1)=27$ and $N_1(g)=N_2(g)=3$.
This polynomial $h$ has $90$ solutions for which $0\leq b_i\leq9$ for $0\leq i\leq8$ and $b_0b_1>0$.
(We attach the last restriction to ensure that the polynomials $F(x)$ that we construct in the subsequent step have nonzero constant and linear terms.)
Each such solution $g(x)$ allows us to create a family of polynomials $F(x)$ with $\{0,1\}$ coefficients and $\deg{F}\leq80$ having $F(1)=27$ and $N_1(F)=N_2(F)=3$.
For example, the polynomial
\[
g(x) = 5 + x + 6x^2 + 5x^4 + x^5 + 5x^6 + 3x^7 + x^8
\]
gives rise to $9^2\binom{8}{4}\binom{9}{3}^2\binom{9}{4}^2 = 635159387520$ qualifying $F(x)$ having the form $1+x+\sum_{k=2}^{80} a_k x^k$ with $\{0,1\}$ coefficients: one selects four of $\{a_9,a_{18},a_{27},\ldots,a_{72}\}$ to set to $1$, and six of $\{a_2,a_{11},a_{20},\ldots,a_{74}\}$, etc.
Constructing all such polynomials $F(x)$ over all $90$ solutions is still prohibitive, since this amounts to about $4.2\cdot10^{17}$ possibilities in all, but we tested a portion of the polynomials $F(x)$ in each case to search for examples where $\abs{N_3(F)}=3$ and $\abs{N_4(F)}=3m$ with $3\nmid m$, so that $\abs{M_{\mathbb{Z}_{81}}(F)} = 3^7m$.
We constructed more than half a million positive $\mathbb{Z}_{81}$-measures of this form by using this procedure.
All $m$ had a prime factor $q\equiv 1 \bmod 81.$

A similar computation was performed for $\mathbb{Z}_{49}$, where we constructed a large family of polynomials with $7\mid F(1)$ and $N_1(F)=7$ in a similar way.
Requiring that a polynomial $g(x)=b_0+b_1 x+\cdots+b_6 x^6$ has $g(1)=7k$ for a small integer $k$ and $N_1(g)=7$ produces a polynomial equation in the coefficients $\{b_0, \ldots, b_5\}$, and we may determine solutions to this equation in nonnegative integers (requiring again that $b_0b_1 > 0$).
There are $82$ such solutions with $g(1)=7$, and $298$ with $g(1)=14$.
As with $\mathbb{Z}_{81}$, each such solution $g(x)$ allowed us to create a family of polynomials $F(x)$ with $\{0,1\}$ coefficients and $\deg{F}\leq48$ having $F(1)=7$ or $14$ and $N_1(F)=7$.
This procedure led us to find more than $14000$ different values of $m$ where $\abs{M_{\mathbb{Z}_{49}}(F)}=7^3m$, $7\nmid m$, is achieved by a polynomial $F(x)$ with $\{0,1\}$ coefficients having $F(0)=F'(0)=1$, degree at most $48$, $F(1)=N_1(F)=7$ and $\abs{N_2(F)}=7m$.
We also constructed more than $51$ million values of $m$ where the $\mathbb{Z}_{49}$-measure $2\cdot7^3m$, $7\nmid m$, is achieved in a similar way by a polynomial $F(x)$ with $F(1)=14$, $N_1(F)=7$ and $\abs{N_2(F)}=7m$.
In both cases every such integer $m$ had at least one prime divisor $q \equiv 1 \bmod 49$.

It is interesting that in both of these groups using only $a_i=1$ in Theorem~\ref{thmpowers} sufficed; larger odd powers of primes were never observed to be necessary.
This was established for $\mathbb{Z}_{25}$ and $\mathbb{Z}_{27}$, as we remarked at the end of Section~\ref{secPf27}.
It seems possible then that a somewhat stronger version of Theorem~\ref{thmpowers} may hold, with further restrictions on the $a_i$.
We leave this question to future research.

\subsection{Computations connected to Theorem~\ref{thmZ25}}\label{subsecCalcs2}

In connection with Theorem~\ref{thmZ25}, we computed $M_{\mathbb{Z}_{25}}(F)$ for all polynomials $F(x)$ with $\deg(F)\leq24$ having $\{-1,0,1\}$ coefficients and $F(1)=5$.
This produced polynomials attaining the measures $5^3q$ for nearly half of the primes $q<5000$ in $\mathcal{T}_1^{(5)}$ with $q=1\bmod5$, including all such primes less than $251$.
Simple representations for the first few primes are
\begin{align*}
11\cdot5^3 &= M_{\mathbb{Z}_{25}}(1+x^4+x^5+x^9+x^{10}),\\
31\cdot5^3 &= M_{\mathbb{Z}_{25}}(-1+x+x^6+x^9+x^{11}+x^{13}+x^{16}),\\
41\cdot5^3 &= M_{\mathbb{Z}_{25}}(1+x^5+x^{10}+x^{14}+x^{15}),\\
61\cdot5^3 &= M_{\mathbb{Z}_{25}}(-1+x+x^6+x^8+x^{11}+x^{12}+x^{16}),\\
71\cdot5^3 &= M_{\mathbb{Z}_{25}}(-1-x+x^4-x^5+x^7+x^9+x^{11}+x^{12}+x^{13}+x^{14}+x^{18}).
\end{align*}

\subsection{Computations connected to Theorem~\ref{thmZ27}}\label{subsecCalcs3}

For Theorem~\ref{thmZ27}, our search in $\mathbb{Z}_{27}$ over polynomials $F(x)$ with $\deg(F)\leq26$ having $\{-1,0,1\}$ coefficients and $F(1)=3$ or $9$ found polynomials attaining each measure $3^4q$ for all primes $q<5000$ in $\mathcal{T}_1^{(3)}$ with $q=1\bmod9$, including
\begin{align*}
19\cdot3^4 &= M_{\mathbb{Z}_{27}}\left(1+x^8+x^9\right),\\
37\cdot3^4 &= M_{\mathbb{Z}_{27}}\left(-1+x^2+x^3+x^{10}+x^{11}\right),\\
109\cdot3^4 &= M_{\mathbb{Z}_{27}}\left(1+x^3+x^4\right).
\end{align*}
It also found representations for $3^5q$ for all primes $q<5000$ in $\mathcal{T}_3^{(3)}$ with $q=1\bmod9$, for example,
\begin{align*}
73\cdot3^5 &= M_{\mathbb{Z}_{27}}\left(1+x^2+x^3+x^4+x^5+x^7+x^9+x^{11}+x^{12}\right),\\
271\cdot3^5 &= M_{\mathbb{Z}_{27}}\left(1+x+x^2+x^3+x^4+x^5+x^6+x^8+x^9\right),\\
307\cdot3^5 &= M_{\mathbb{Z}_{27}}\left(-1+x+x^4+x^5+x^6+x^7+x^8+x^{10}+x^{12}+x^{14}+x^{15}\right),
\end{align*}
as well as $3^5q$ for the smallest $q=1\bmod3$ in $\mathcal{T}_1^{(3)}$:
\begin{align*}
7\cdot3^5 &= M_{\mathbb{Z}_{27}}\left(1+x^2+x^3+x^5+x^6+x^8+x^9+x^{11}+x^{12}\right),\\
13\cdot3^5 &= M_{\mathbb{Z}_{27}}\left(1+x^3+x^6+x^9+x^{12}+x^{14}+x^{15}+x^{17}+x^{18}\right),\\
19\cdot3^5 &= M_{\mathbb{Z}_{27}}\left(1+x^3+x^6+x^9+x^{12}+x^{15}+x^{18}+x^{20}+x^{21}\right),\\
31\cdot3^5 &= M_{\mathbb{Z}_{27}}\left(-1-x+x^2-x^3+x^5-x^6+x^7+x^8+x^{10}+x^{11}\right.\\
&\qquad\qquad\quad\left.+x^{12}+x^{13}+x^{14}+x^{17}+x^{20}+x^{23}+x^{26}\right).
\end{align*}

\section*{Acknowledgements}

We thank Michel Marcus for pointing out the connection between $\mathcal{T}^{(5)}_2$ and the artiads.

\end{document}